\documentclass[11pt,reqno]{amsproc}
\usepackage[margin=1in]{geometry}
\usepackage{amsmath, amsthm, amssymb}
\usepackage{times, esint, stackrel, enumitem, color, hyperref, graphicx}
\usepackage{enumitem}

\theoremstyle{plain}
\newtheorem{thm}{Theorem}
\newtheorem{prop}{Proposition}
\newtheorem{lemma}{Lemma}

\theoremstyle{definition}

\newtheorem{rem}{Remark}

\newcommand{\be}{\begin{equation}}
\newcommand{\ee}{\end{equation}}
\newcommand{\bea}{\begin{eqnarray}}
\newcommand{\eea}{\end{eqnarray}}
\newcommand{\ve}{{\varepsilon}}

\newcommand{\Rr}{\mathbb{R}}

\newcommand{\ol}[1]{\mkern 1.5mu\overline{\mkern-1.5mu#1\mkern-1.5mu}\mkern 1.5mu}

\newcommand{\red}[1]{{#1}}
  
\newcommand{\supp}{\rm supp}

\title[Onsager's Conjecture on Domains with Boundary]{Onsager's conjecture and anomalous dissipation on domains with boundary}
\author{Theodore D. Drivas and Huy Q. Nguyen}
\address{Department of Mathematics, Princeton University, Princeton, NJ 08544}
\email{tdrivas@math.princeton.edu}
\address{Department of Mathematics, Princeton University, Princeton, NJ 08544}
\email{qn@math.princeton.edu}
\date{today}

\begin{document}
\begin{abstract} 
We give a localized regularity condition for energy conservation of weak solutions of the Euler equations on a  domain $\Omega\subset \mathbb{R}^d$, $d\ge 2$, with boundary. In the bulk of fluid, we assume Besov regularity of the velocity { $u\in L^3(0,T;B_{3}^{1/3, c_0})$}. On an arbitrary thin neighborhood of the boundary, we assume boundedness of velocity and pressure and, at the boundary, we assume continuity of wall-normal velocity. 
We also prove two theorems which establish that the global viscous dissipation vanishes in the inviscid limit for Leray--Hopf solutions $u^\nu$ of the Navier-Stokes equations under the similar assumptions, but holding uniformly in a thin boundary layer of width \red{$O(\nu^{\min\{1,\frac{1}{2(1-\sigma)}\}})$}  {when $u\in L^3(0, T; B_3^{\sigma, c_0})$ in the interior for any $\sigma\in [1/3,1]$}. The first theorem assumes continuity of the velocity in the boundary layer whereas the second assumes a condition on the vanishing of energy dissipation within the layer.  \red{In both cases, strong $L^3_tL^3_{x,loc}$ convergence holds to a weak solution of the Euler equations}. Finally, if a strong Euler solution exists in the background, we show that equicontinuity at the boundary within a $O(\nu)$  strip alone suffices to conclude the absence of anomalous dissipation.
\end{abstract}

 \maketitle
 
 In his celebrated 1949 paper, Lars Onsager \cite{O49} announced that spatial H\"{o}lder exponents $\leq 1/3$ are required of the velocity field to sustain non-vanishing energy dissipation in the limit of zero viscosity
  \be\label{zerothLaw3d}
  \lim_{\nu\to 0}  \nu\langle |\nabla u^\nu|^2\rangle =  \varepsilon>0
 \ee
  where  $\langle \cdot \rangle$ is some relevant averaging procedure, space, time or ensemble.
For flows on domains without boundary (such as $\mathbb{T}^d$ or $\mathbb{R}^d$), Onsager's assertion has since been proved \cite{GLE94,CET94,DR00,CCFS08,DE17} and dissipative Euler solutions with less regularity have been constructed \cite{LS09,LS10,LS12,I16,BLSV17}.  See the recent review \cite{E18}.

Although there is a wealth of corroboratory evidence  from numerical simulations for anomalous dissipation for flows on the torus \cite{BO95,TBS02,KRS98,KIYIU03},  all empirical evidence from laboratory experiments involve flows confined by solid boundaries.  The most common experiments study turbulent flows produced downstream of wire-mesh grids in wind-tunnels or 
generated by flows past other solid obstacles, such as plates or cylinders.  Such experiments indeed indicate the presence of a dissipation anomaly;  the data plotted in 
\cite{KRS84,PKW02} clearly shows that $\langle \ve^\nu(t) \rangle$ is nearly independent 
of $\nu=1/Re$ as $Re$ increases.

 The role of  solid walls for anomalous energy dissipation is complex.  Vorticity is fed into flows by viscous boundary layers  that detach from the walls, thereby generating turbulence. These boundary layers become thinner as $\nu=1/Re$ decreases, and sharp velocity gradients may emerge and propagate into the bulk.
   In this way, walls may be the source of the irregular velocity fields in the interior required, according to Onsager, to sustain anomalous dissipation. We emphasize that singularities need not be produced in finite time by Euler, although this is also a possibility, see Remark 3 of \cite{DE17}. This latter point is particularly relevant in two-dimensions, where global Euler solutions exist for regular enough initial conditions.  Walls could also play a more active role in dissipating energy. It is conceivable, for example, that anomalous dissipation has contributions not only from the bulk but also from the boundary layer itself.   This latter issue is closely connected with the formation of singularities in finite time for the Euler equation and elucidates the need for studying weak solutions.

Indeed, assuming the existence of a smooth Euler solution satisfying no-flow boundary conditions, Kato  \cite{Kato84} proved that the following two conditions are equivalent: (i) the integrated energy dissipation
vanishes in a very thin boundary layer of thickness $O(\nu)$ and
(ii) any Navier-Stokes solution with no slip boundary conditions at the wall converges strongly in $L^\infty_tL_x^2$ to the
Euler solution as $\nu\to 0$. The latter statement implies that energy dissipation must, in fact,
vanish everywhere in the domain.    However, present experiments suggest that, at least in certain  three-dimensional geometries, that energy dissipation vanishes very near the boundary but total dissipation does not.   In Taylor-Couette cells with smooth walls, Cadot et al. \cite{C97},  found that energy dissipation in the boundary layer of such flows decreases with Reynolds number but that the energy dissipation in the bulk appears to be independent of Reynolds number. 
This observation is consistent with the Prandtl--von-K\'{a}rm\'{a}n theory of the logarithmic-law-of-the-wall for pipe and channel flow which predicts \cite{P32,vK30} that the energy dissipation vanishes as $\nu\to 0$ within the viscous sublayer $O(\nu/u_*)$ where $u_*$ is the friction velocity.  Coupled with Kato's result\footnote{Standard scaling theory (see e.g. \cite{S07,P07}) implies that, if the the channel or pipe flow is forced with centerline velocity $u_c$ fixed independent of $Re_*$, then $u_*\sim u_c/{\rm log}(Re_*)$ where $Re_*= hu_*/\nu$ is the Reynolds number based on friction velocity $u_*$ and half-height $h$ of the channel/pipe. If this holds asymptotically as large $Re_*$, it follows that the viscous sublayer of width $O(\nu/u_*)$ is asymptotically larger than the layer of width $O(\nu)$ and therefore the Cadot et al. experiment indicates that energy dissipation vanishes in Kato's boundary layer.  See Chapter VI \S C of \cite{Enotes}  for further discussion.}, these observations serve as indirect evidence for finite-time singularity formation for the $3d$ Euler equations in domains with boundary.  A weaker notion of Euler solution must be considered if it is to describe the inviscid limit. 

On the other hand, in two-dimensions, simulations of the Navier-Stokes equation with critical Navier--slip boundary conditions (slip length proportional to $\nu$) of a vortex dipole crashing into a wall indicate the persistence of energy dissipation in the zero-viscosity limit \cite{Farge11}.  Additional details of their observations are provided in the discussion after our Theorem \ref{thm:inviscid}.  This is a setting in which a global smooth Euler solution exists, but nevertheless there appears to be a dissipative anomaly. 
{We remark that, unlike the situation on the torus \cite{DE17}, convergence to a dissipative weak solution and the existence of a strong Euler solution with the same data/forcing are not in contradiction provided, of course, that there is anomalous dissipation at least within a Kato--layer of the boundary.  This is possible due to the failure of the classical weak-strong uniqueness theorem on domains with solid boundaries.  Instead, weak solutions coincide with strong solutions if the latter exist and the former obeys the additional condition that it is continuous in a neighborhood of the boundary, see Theorem 5.3 of  \cite{W17}. There are examples demonstrating the sharpness of this result \cite{B14}. }

In domains without boundary, Onsager conjectured that dissipative \emph{weak} Euler solutions should describe the high-$Re$ limit of turbulent solutions. On bounded domains a similar conjecture can be made.  Indeed,  Constantin \& Vicol \cite{CV17} show that, in two dimensions, interior enstrophy bounds uniform in viscosity  are sufficient to conclude convergence weakly in $L^2(0,T;L^2(\Omega))$ to a weak Euler solution\footnote{In the paper \cite{CV17}, convergence is obtained at the level of the vorticity equation.  However, inspection of their argument shows that convergence of the velocity is implied.}.   Such solutions may be dissipative.  In three-dimensions, they show that weak convergence in $L^2(\Omega)$ for a.e. $t$ together with some uniform bounds on the second order structure function suffice to conclude convergence.   In any dimension for $\Omega=\mathbb{T}^d$,  \cite{DE17,Isett17} show that uniform boundedness of $u^\nu$ in the space $L^2(0,T;B_2^{s,\infty}(\Omega))$  for any $s>0$  implies strong convergence in $L^2([0,T]\times \Omega)$ to a weak Euler solution (at least along subsequences) in the limit of zero-viscosity.   See Lemma \ref{lemm:inviscid} in \S \ref{secEnCons} of the present note for similar conditions on arbitrary bounded domains.  
 This Besov regularity is equivalent to the physically reasonable conditions that the solutions $u^\nu$ have finite energy uniformly in $\nu$ and that there exists a uniform upper bound for the second-order structure function $S_2^\nu(r) := \langle |\delta u^\nu(r;\cdot)|^2\rangle $ of the form $S_2^\nu(r) \leq C |r|^{2s}$ that holds for all scales $|r|$.  See Remark 1 of \cite{DE17}.  This remains an open and active area of research.

 Thus, it appears that understanding the behavior of weak Euler solutions is crucial for the study of the inviscid limit in bounded domains.  Accordingly, in this note, we discuss the extension of Onsager's ideas in the presence of solid boundaries.   Let $d\geq 2$ and $\Omega\subset \mathbb{R}^d$ be a bounded domain with  $C^2$ boundary, $\partial \Omega$. The incompressible Euler equations read
\begin{align} \label{Eeq}
   \partial_tu+\nabla \cdot (u\otimes u)+\nabla p&=0,\qquad \text{in} \ \Omega\times (0,T),\\  \label{Eincomp}
  \nabla \cdot u&=0,\qquad \text{in} \  \Omega\times (0,T),\\
  u\cdot \hat{n} &= 0,\qquad \text{on} \ \partial \Omega\times (0,T),\label{Ebc}
\end{align}
where $ \hat{n}= \hat{n}(x)$ is the unit normal vector field to the boundary $\partial\Omega$.
We say that  $(u,p)$ is a weak Euler solution on $\Omega\times (0,T)$  if $u\in C_w(0,T; H(\Omega))$, $p\in L^1_{loc}(\Omega\times (0, T))$ and 
\be\label{weakEuler}
 \int_0^T\int_\Omega  \big(u \cdot  \partial_t\varphi   +  u \otimes u  :  \nabla  \varphi  +p \nabla \cdot  \varphi \big)dxdt =0
\ee
for all test vector fields $\varphi\in C_0^\infty(\Omega\times (0,T))$.
Here, $H(\Omega)$ denotes the space of $L^2$ divergence--free vector fields which are parallel to the boundary $\partial\Omega$, i.e. $H(\Omega)$ is the completion in $L^2(\Omega)$ of the space  
\[
\left\{w\in C^\infty_c(\Omega; \Rr^d): \nabla\cdot w=0\right\}.
\]
Note that for any $f\in H(\Omega)$, we have that $ f\cdot \hat{n} \in H^{-1/2}(\partial\Omega)$ so that the boundary condition \eqref{Ebc} is well defined in that space.

 We show that any weak Euler solution possessing a certain degree of smoothness in the interior and having \emph{continuous wall-normal} velocity and \emph{bounded} velocity and pressure arbitrarily close to the walls must conserve kinetic energy globally.  Regularity in the interior will be measured in Besov spaces, following the classical work of Constantin-E-Titi \cite{CET94}. 
Recall that the Besov space $B_p^{\sigma,\infty}(U)$ is made up of measurable functions $f:U\subset \Rr^d\to \mathbb{R}$ which are finite in the norm
\be
\| f\|_{B_p^{\sigma,\infty}(U)}:= \| f\|_{L^p(U)} 
+ \sup_{r\in \Rr^d}\frac{\|f(\cdot+r) - f(\cdot)\|_{L^p(U\cap (U-\{r\}))}}{|r|^\sigma}
\label{besov-def} \ee
for $p\geq 1$ and $\sigma\in (0,1)$. These spaces are very relevant from a physical point of view, as they can be described simply in terms of scalings of structure functions $S_p(r):=\langle |\delta u(r;\cdot)|^p\rangle$.  See Remark 1 of \cite{DE17} and discussion in \cite{GLE95}.  Moreover, velocity fields in the high-$Re$ limit are not expected to have anything near a third of a derivative when measured in $L^\infty$ (H\"{o}lder continuous $C^\alpha$ with $\alpha\approx1/3$).  Indeed, this regularity would correspond to Kolmogorov's celebrated 1941 mono-fractal theory of turbulence. Instead, turbulent flows are observed to be intermittent and exhibit a multi-fractal spectrum, see e.g. \cite{S96,A84} and high--$Re$ velocity fields have remarkable close to a third of a derivative when measuring in $L^3$, i.e. they are nearly $B_3^{\sigma,\infty}$ with $\sigma \lessapprox 1/3$ (as indicated by inertial range scaling of the third-order absolute structure function).  

{As shown in \cite{CCFS08}, energy conservation for the Euler equations in the whole space $\Rr^d$ holds provided that $u\in B_3^{1/3, c_0}(\Rr^d)$, a subspace of $B_3^{1/3, \infty}(\Rr^d)$ that can be defined as follow
\begin{equation}\label{def:Bc}
B_p^{\sigma, c_0}(U)=\left\{f\in L^p(U):~\lim_{|r|\to 0}\frac{\|f(\cdot+r) - f(\cdot)\|_{L^p(U\cap (U-\{r\}))}}{|r|^\sigma}=0\right\}.
\end{equation}
Let us note that $B_p^{\sigma', \infty}(U)\subset B_p^{\sigma, c_0}(U)\subset B_p^{\sigma, \infty}(U)$ for any $\sigma'>\sigma$. Moreover, \cite{CCFS08} provides an example of a solenoidal vector field in $B_3^{1/3, \infty}(\mathbb{R}^3)$ for which the energy flux does not vanish, showing the sharpness  of their kinematic argument. We refer to \cite{Shvydkoy} for the proof of local energy balance in the class $B_3^{1/3, c_0}(\mathbb{R}^d)$. For time-dependent functions we define
\be\label{def:Bct}
L^q(0, T; B_p^{\sigma, c_0}(U))=\left\{f\in L^q(0, T; L^p(U)):~\lim_{|r|\to 0}\frac{\|f(\cdot+r) - f(\cdot)\|_{L^q(0, T; L^p(U\cap (U-\{r\})))}}{|r|^\sigma}=0\right\}.
\ee
}
Denote the distance function to the boundary $d(x)$ and near-wall strip $\Omega_a$ by
\be
d(x)=\inf_{y\in \partial\Omega} |x-y| \quad\text{and}\quad \Omega_a=\{x\in \Omega: d(x)<a\}.
\ee
We now state our first theorem regarding weak solutions of the Euler equations.  

\begin{thm}\label{thm:onsager}
Let  $(u,p)$ be a weak solution of the Euler equations \eqref{Eeq}--\eqref{Ebc} on $\Omega\times (0, T)$. Assume that 
\begin{itemize}
\item for any $U\Subset \Omega$ we have
{\be\label{bndassum1}
u\in L^3(0,T;B_{3}^{1/3, c_0}(U)),
\ee
}
\item for some $\ve>0$ arbitrarily small we have
\be\label{bndassum2}
u \in L^3(0,T; L^\infty(\Omega_\ve)),\quad p \in L^{3/2}(0,T;L^\infty(\Omega_\ve))
\ee
and continuity of the wall-normal velocity at the boundary
\be\label{bndassum3}
\lim_{\delta\to 0}\sup_{x\in \Omega_{\delta}} |\hat{n}(\pi(x)) \cdot u (x,t)| = 0 \qquad \text{in} \qquad L^3((0,T))
\ee
where $\pi(x)$ denotes the unique closest point to $x$ on $\partial\Omega$ (see Eq. \eqref{distDeriv}).
\end{itemize}
 Then energy is globally conserved
\be\label{energycons}
\frac{1}{2} \int_{\Omega} |u(x,t)|^2 dx = \frac{1}{2} \int_{\Omega} |u(x, 0)|^2 dx  \qquad \forall t\in [0,T).
\ee
\end{thm}

\begin{rem}
Theorem \ref{thm:onsager} extends the results of \cite{CET94} to the bounded domain setting and refines the recent results of Bardos-Titi \cite{BT18} where the velocity $u$ was assumed to be uniformly $C^\alpha$ on $\ol\Omega$ with $\alpha>\frac{1}{3}$. Here, our localized condition is compatible with \cite{CET94} in the bulk of fluid, and is much weaker than \cite{BT18} near the boundary.   Independently, Bardos, Titi and Wiedemann arrived at this result as an implication of a local energy equality for sufficiently regular weak Euler solutions -- see Theorem 4.1 and Remark 4.2 \cite{BTW18}. Their result is stated with interior regularity measured in H\"{o}lder spaces $C^\sigma$, $\sigma>1/3$, whereas in our Theorem we use Besov spaces {$B^{1/3, c_0}_3$} and provide a different proof.  From them, we learned of the sufficiency of continuity  at the boundary  for the normal component of the velocity alone, rather than the full field (private communication).
\end{rem}

\begin{rem}\label{contatBnd}
Condition \eqref{bndassum3} could be replaced  by the \emph{refined} continuity at the boundary of $ \hat n\cdot u$ in the region $\Omega_\ve$.  By this condition, we mean specifically that there exist a function 
\be
0 \leq \gamma(r,t)\in L^3( 0,T;L^1_{loc}(\mathbb{R}^+))
\ee
with the properties that 
\begin{enumerate}[label=(\roman*)]
\item for any $\delta>0$ there exists $\rho=\rho(\delta)$ such that 
\be
|\hat{n}(\pi(x))\cdot u(x,t)|\leq \delta \gamma(|x-y|,t) 
\ee  
for all $(x,t) \in (\Omega_{\ve}\cap\Omega_\rho) \times [0,T]$
\item and that
 \be\label{gamBnd}
\frac{1}{h}\int_{0}^{h} |\gamma(y,t)|^3dy  \leq  \Gamma(t) + o_h(1)\quad\text{in}~L^1((0,T))
\ee
for all $h$ sufficiently small.
\end{enumerate}
See the proof of Theorem \ref{thm:inviscid}, where we implement this assumption.  Additionally, from inspection of the proof of Theorem \ref{thm:onsager}, instead of the continuity at the wall \eqref{bndassum3}, we could  require  the weaker condition on the near-wall velocity
\be\label{alterCond}
\quad \lim_{h \to 0} \lim_{\ell\to 0}\left[\frac{1}{\ell}\sup_{|r|<\ell}\int_{\Omega_h\setminus\Omega_{h-\ell}} \big|\hat n(\pi(x+r))\cdot u(x+r,\cdot)\big| dx\right] = 0 \quad \text{in} \quad L^3((0,T)).
\ee
\end{rem}

Furthermore, $u\in C^\delta(\Omega_\ve)$ for any $\delta\in(0,1)$ regularity suffices to guarantee \eqref{bndassum2},\eqref{bndassum3} as we now note.
\begin{rem}\label{rema:pressure}
Theorem \ref{thm:onsager} is stated with separate assumptions on the pressure.  Of course, the pressure is  completely determined from the velocity via the Poisson equation with Neumann boundary conditions
\begin{align}\label{presspoisson}
\qquad -\Delta p &= \nabla \otimes \nabla : (u\otimes u), \qquad \text{in} \ \Omega,\\
 -\partial_{\hat{n}}p &= \nabla \hat{n}: (u\otimes u), \qquad \quad \  \text{on} \ \partial\Omega.
\end{align}
Therefore, one would like to make assumptions only on the velocity field and not the pressure separately.  Due to the criticality of $L^\infty$ for elliptic regularity theory, we are not able to accomplish this only assuming continuity.
However, from the assumptions of Theorem \ref{thm:onsager} we have $u\in L^3(0, T; L^3(\Omega))$ and $u\vert_{\partial\Omega}\in L^3(0, T; L^\infty(\partial\Omega))$.  Using the method of {\it transposition} 
 it follows that $p\in L^{3/2}(0, T; L^{3/2}(\Omega))$ (see Lemma 2 in \cite{EscMon}).    Combining Proposition 1.2 of \cite{BT18} with a localization argument shows that for any $\delta\in (0, 1)$ we have
\be\label{elliptic}
u \in L^3(0,T; C^\delta(\ol{\Omega_\ve}))\cap L^3(0,T; L^3(\Omega))\quad \implies \quad p  \in L^{3/2}(0,T; C^\delta(\ol{\Omega_{\ve/2}})).
\ee
The conditions in \eqref{bndassum2},\eqref{bndassum3} are implied by the single condition $u \in L^3(0,T; C^\delta(\ol{\Omega_\ve}))$ for arbitrarily $\delta>0$.
\end{rem}

  Our next theorem applies to the inviscid limit of   Leray--Hopf solutions of the Navier-Stokes equations with no-slip boundary conditions.  Specifically, we consider $u^\nu\in C_w(0,T; L^2(\Omega))\cap L^2(0,T; H_0^1(\Omega))$
  satisfying
  \begin{align} \label{NSeq}
   \partial_tu^\nu+\nabla\cdot (u^\nu \otimes u^\nu)+\nabla p^\nu&=\nu \Delta u^\nu,\qquad \text{in} \ \Omega\times (0,T),\\
  \nabla \cdot u^\nu&=0,\qquad\qquad \text{in} \  \Omega\times (0,T),\\
  u^\nu &= 0,\qquad\qquad   \text{on} \ \partial \Omega\times (0,T),\label{NSbc}
\end{align}
  in the sense of distributions, i.e.
  \be\label{weakNS}
 \int_0^T\int_\Omega  \big(u^\nu \cdot  \partial_t\varphi   +  u^\nu \otimes u^\nu  :  \nabla  \varphi  +\nu u^\nu\cdot \Delta \varphi)dxdt =0
\ee
 for all divergence-free test vector fields $\varphi\in C_0^\infty(\Omega\times (0,T))$. Such solutions are known to exist globally in time and satisfy the  energy inequality \cite{L34,H50}.  In fact, a pressure $p^\nu$ can be reintroduced from the weak formulation \eqref{weakNS}, see Theorem 1 of \cite{Jmm14} (Theorem \ref{reintro:p} herein). It is a solution of \eqref{presspoisson} with the Neumann boundary conditions
$\partial_{\hat{n}}p^\nu = \nu \hat{n}\cdot \Delta u^\nu$ on $\partial\Omega$.

 We show that if conditions similar to \eqref{bndassum1}, \eqref{bndassum2} and \eqref{bndassum3} hold uniformly in viscosity in appropriate regions, then global viscous dissipation vanishes in the limit $\nu\to 0$ for any sequence of Leray--Hopf solutions \cite{L34,H50}.  Furthermore, the limiting velocity field is a weak solution of the Euler equations.
  
\begin{thm}\label{thm:onsagerNS}
 Let $\Omega\subset \mathbb{R}^d$ be a bounded domain with  $C^2$ boundary, $\partial \Omega$.
Consider a sequence $\{u^\nu\}_{\nu>0}$ of Leray-Hopf weak solutions of the Navier-Stokes equations \eqref{NSeq}--\eqref{NSbc} with initial data $u_0^\nu\to u_0$  strongly in $L^2(\Omega)$.    Let {$\sigma\in[1/3,1]$} and \red{$\beta=\min\{1,\frac{1}{2(1-\sigma)}\}$}. Assume  that
\begin{align}\label{bndassum1NS}
&\text{$u^\nu$ are uniformly in $\nu$ bounded in\   $L^3(0,T;B_{3}^{\sigma,\red{c_0}}(\Omega\setminus \Omega_{\nu^\beta/2}))$,}\\
 \label{bndassum2NS}
&\text{$u^\nu$ are uniformly in $\nu$ bounded in $L^{3}(0,T;L^\infty(\Omega_{\nu^\beta}))$,}\\ 
 &\text{$p^\nu$ are uniformly in $\nu$ bounded in $L^{3/2}(0,T;L^\infty(\Omega_{\nu^\beta}))$}, \label{bndassum3NS}\\ 
 \nonumber
 &\text{$u^\nu$ is equicontinuous at $\partial \Omega$ in the region $\Omega_{\nu^\beta}$, i.e.}\\
 &  
\qquad\quad  \lim_{\nu\to 0}\sup_{x\in \Omega_{\nu^\beta}} |u^\nu(x,t)| = 0 \ \text{in} \ L^3((0,T)). \label{bndassum4NS}
 \end{align}
Then the global viscous dissipation vanishes 
\be\label{noanomaly}
\lim_{\nu \to 0}\nu \int_0^T \int_\Omega  |\nabla u^\nu (x, s)|^2 dx ds = 0.
\ee
Upon extracting a subsequence, $u^\nu \to u$ strongly in $L^3(0,T; L^3_{{loc}}(\Omega))$ where $u$ is a weak Euler solution.
\end{thm} 

\begin{rem}
Assumption \eqref{bndassum4NS} could be replaced by  \emph{refined equicontinuity at the boundary} of  $ u^\nu$ in the region $\Omega_{\nu^\beta}$, see Remark \ref{contatBnd}.   Moreover, as for Euler, we could also replace  \eqref{bndassum4NS} with  
\be\label{alterCond2}
\quad \lim_{\nu \to 0} \left[\frac{1}{\ell}\sup_{|r|<\ell}\int_{\Omega_h\setminus\Omega_{h-\ell}} \big| u^\nu(x+r,\cdot)\big| dx\right] \Bigg|_{h,\ell \sim \nu^\beta} = 0 \quad \text{in} \quad L^3((0,T)).
\ee
\end{rem}
\begin{rem}
In  \cite{BTW18}, Bardos, Titi and Wiedemann  prove  a similar result -- Theorem 5.1 therein. The main difference is that, in their work, equicontinuity of the wall-normal velocity and boundedness of the solution fields are assumed within a fixed strip  (independent of $\nu$), whereas our approach allows us to make the assumption in a boundary layer that shrinks with viscosity. Equicontinuity of the full (rather than wall-normal) velocity field at the boundary \eqref{bndassum4NS} is required in our theorem precisely because we make the assumptions in a vanishingly thin boundary layer. Our proof can be easily modified to assume only continuity of the wall-normal velocity alone (as well as boundedness) in a fixed region around the boundary by taking the limit $\nu\to 0$ first in the proof of Theorem \ref{thm:onsagerNS}, and $\ell,h\to 0$ subsequently.  
\end{rem}

We now provide a variant of the previous theorem which assumes vanishing of the near-wall dissipation 
\begin{thm}\label{thm:onsagerNS'}
 Let $\Omega\subset \mathbb{R}^d$ be a bounded domain with  $C^2$ boundary, $\partial \Omega$.
Consider a sequence $\{u^\nu\}_{\nu>0}$ of Leray-Hopf weak solutions of the Navier-Stokes equations \eqref{NSeq}--\eqref{NSbc} with initial data $u_0^\nu\to u_0$  strongly in $L^2(\Omega)$.     Let \red{$\sigma\in[1/3,1]$} and \red{$\beta=\min\{1,\frac{1}{2(1-\sigma)}\}$}. Assume  that
 \begin{align}\label{bndassum1NS'}
&\text{$u^\nu$ are uniformly in $\nu$ bounded in\   $L^3(0,T;B_{3}^{\sigma,\red{c_0}}(\Omega\setminus \Omega_{\nu^\beta/2}))$,}\\
 \label{bndassum2NS'}
&\text{$u^\nu$ are uniformly in $\nu$ bounded in $L^\infty(0,T;L^\infty(\Omega_{\nu^\beta}))$,}\\ 
 &\text{$p^\nu$ are uniformly in $\nu$ bounded in $L^2(0,T;L^\infty(\Omega_{\nu^\beta}))$}, \label{bndassum3NS'}\\ 
 &\lim_{\nu\to 0} \nu \int_0^T\int_{\Omega_{\nu^\beta}} |\nabla u^\nu(x,s)|^2dx ds =o_\nu(\nu^{1-\beta}),\label{cd:vanish}
\end{align}
where the notation $f(\nu) = o_\nu(g(\nu))$ denotes $f(\nu)/g(\nu)\to 0$ as $\nu\to 0$.  Then the dissipation vanishes 
\[
\lim_{\nu \to 0}\nu \int_0^T \int_\Omega  |\nabla u^\nu (x, s)|^2 dx ds = 0.
\]
Upon extracting a subsequence, $u^\nu \to u$ strongly in $L^3(0,T; L^3_{{loc}}(\Omega))$ where $u$ is a weak Euler solution.
\end{thm}

\begin{rem}
In Theorems \ref{thm:onsagerNS} and \ref{thm:onsagerNS'}, for all $\sigma\in[1/2,1]$, the layer  $\Omega_{\nu^\beta}$ becomes the Kato width $O(\nu)$. 
\end{rem}

Given that some laboratory experiments and numerical simulations observe anomalous dissipation,  it is natural to ask which of the assumptions \eqref{bndassum1NS}--\eqref{bndassum4NS} of Theorems \ref{thm:onsagerNS} or \eqref{bndassum1NS'}--\eqref{cd:vanish}  of Theorem \ref{thm:onsagerNS'} being violated is responsible for the observed anomalies.  
On the torus, the failure of uniform Besov regularity is required for dissipation to persist in the limit $\nu\to 0$ \cite{CET94,DE17}.  In domains with boundaries, observations of inertial range structure function scalings in the bulk  \cite{LS99} show similarities with the torus and indicate that such uniform regularity may fail for arbitrarily small viscosity.  On the other hand, the violation of  \eqref{bndassum2NS}--\eqref{bndassum4NS} or \eqref{bndassum2NS'}--\eqref{cd:vanish} provide new mechanisms for anomalous dissipation on domains with solid walls.   Indeed, it is possible that emergent discontinuities in a near-wall boundary layer  \eqref{bndassum4NS}, or failure for the near-wall dissipation to vanish as $\nu\to 0$  \eqref{cd:vanish} is more relevant for observations of anomalous dissipation in the simulation of a vortex dipole crash \cite{Farge11}\footnote{Recall that the simulations of  \cite{Farge11} are for Navier-Stokes with degenerate slip boundary conditions (slip length $\propto \nu$) rather than the no-slip conditions assumed for our Theorems \ref{thm:onsagerNS} and \ref{thm:onsagerNS'}.}. On the other hand, we recall that the experiments of Cadot et al. \cite{C97}  provide evidence of global dissipation anomaly although energy dissipation in the boundary layer of $O(\nu)$ does decrease as Reynolds number increases.  Our assumption \eqref{cd:vanish} implies the latter, namely it implies that
\be\label{katoDiss}
\lim_{\nu\to 0} \nu \int_0^T\int_{\Omega_{\nu}} |\nabla u^\nu(x,s)|^2dx ds =o_\nu(1).
\ee
 Ideally, we would like to replace the condition \eqref{cd:vanish} in  Theorem \ref{thm:onsagerNS'} by the weaker assumption \eqref{katoDiss}.  If \eqref{katoDiss} holds for a given flow (for example, as in the experiment \cite{C97}), such a Theorem would imply that a global dissipation anomaly is only possible if instead one of the boundedness assumptions \eqref{bndassum1NS'}--\eqref{bndassum3NS'} is violated.    
 \vspace{2mm}
 
 The condition \eqref{katoDiss} is the object of Kato's celebrated theorem \cite{Kato84} and is equivalent to convergence (in the energy norm) to the strong Euler solution if it exists.  It is worthwhile to compare the result of Theorems \ref{thm:onsagerNS} and \ref{thm:onsagerNS'}  to those of Kato \cite{Kato84} in more detail.   Our theorems provide sufficient conditions for the integrated energy dissipation to vanish which, unlike Kato's result, do not rely on the existence of a smooth Euler solution.   Indeed, as discussed above, the existing experimental evidence suggests no such strong solution exists in three-dimensions, at least in certain geometries.  However, if a strong Euler solution does exists, then any sequence of Navier-Stokes solutions satisfying the our conditions of Theorem  \ref{thm:onsagerNS} or  \ref{thm:onsagerNS'} must, by Kato  \cite{Kato84}, converge strongly in $L^\infty_tL_x^2$ to that smooth solution.  Note that  condition \eqref{cd:vanish} of Theorem  \ref{thm:onsagerNS'}  for any \red{$\sigma\in [1/3,1]$} already implies vanishing of the dissipation within the Kato layer \eqref{katoDiss}. On the other hand, in the case of two-dimensions -- where global strong Euler solutions are guaranteed to exist -- it is known on the half-space that some uniform equicontinuity at the boundary alone is enough to conclude convergence  \cite{CEIV17}.   This suggests that the sufficient conditions of  \eqref{bndassum1NS}--\eqref{bndassum4NS} of Theorem \ref{thm:onsagerNS} are also much stronger than necessary when a smooth Euler solution exists in the background. In our final theorem, we confirm this expectation by providing a generalization of the result of \cite{CEIV17} for arbitrary domains with smooth boundary in any dimension.

  \begin{thm}\label{thm:inviscid}
 Let $\Omega\subset \mathbb{R}^d$ be a bounded domain with  $C^3$ boundary, $\partial \Omega$.  Suppose there exists a  unique strong solution  $v\in C^1([0,T]\times\ol{\Omega})$ of the Euler equations on $[0,T]\times \Omega$ with $u_0\in L^2(\Omega)$.  
Consider a sequence $\{u^\nu\}_{\nu>0}$ of Leray-Hopf weak solutions of the  Navier-Stokes equations \eqref{NSeq}--\eqref{NSbc} with initial data $u_0^\nu\to u_0$ strongly in $L^2(\Omega)$.    Assume that
\be\label{inviscidLimAssum}
\begin{aligned}
&u^\nu  \ \text{ is equicontinuous at} \ \ \partial\Omega \ \text{in the region } \ \Omega_\nu \ \text{ measured in}\  \ L^3((0,T)).
 \end{aligned}
\ee  
Then the strong inviscid limit holds in the energy norm, i.e. $u^\nu \to v$ strongly in $L^\infty(0,T; L^2(\Omega))$.
\end{thm} 

We use the definition of refined equicontinuity at the boundary for this theorem, see Remark \ref{contatBnd}. 
\vspace{2mm}

   Under the conditions of Theorem \ref{thm:inviscid}, there can be no anomalous dissipation as a consequence of Kato's criteria \cite{Kato84}.  Of course, in two dimensions this theorem applies without the extra assumption of existence of a strong Euler solution whenever  $u_0\in H^s(\Omega)$ for $s>2$.  Therefore, in order for the strong inviscid limit to fail to hold in this setting, discontinuities must develop in a thin `Kato-type' boundary layer near the wall.  This could conceivably occur if singular vortex sheets are formed and shed in the limit $\nu\to 0$. Indeed,  Nguyen van yen, Farge, and Schneider \cite{Farge11} observed exactly this in a careful numerical study of the inviscid limit of 2d Navier-Stokes solutions with critical slip conditions.  In particular,  they find that energy dissipation is concentrated within a vortex sheet of thickness proportional to $\nu$ in the neighborhood of the boundary, and remains so as the sheet rolls up into a spiral and detaches from the wall.  They measure the cumulative dissipation with the thinning sheet and see that it becomes independent of $\nu$ for large Reynolds number, indicating the presence of a dissipative anomaly.  As discussed above, it is an open question whether or not weak Euler solutions can describe  the inviscid limit  and explain these observations (at least for critical Navier--slip and no-slip boundary conditions).  
{\begin{rem}
Theorems \ref{thm:onsagerNS'} is stated for any dimension though the most interesting applications are in 3D.  Theorem \ref{thm:inviscid} is interesting for both 2D and 3D in the presence of a background strong Euler solution. Stronger results in two-dimensions are known without boundary.  In particular, there is an interesting result in \cite{CFLS} showing that if any \emph{weak} Euler solution is obtained in the vanishing viscosity limit of strong Navier-Stokes solutions in $\mathbb{T}^2$, then it conserves the energy provided only that the initial vorticity is $\omega_0\in L^p(\mathbb{T}^2)$ with $p>1$.   Recall that in 2D, in view of  the scaling invariant embedding $W^{1, 3/2}\subset B^{1/3, \infty}_{3}$, the Onsager threshold corresponds to the integrability $L^p$ with $p>3/2$ for the vorticity.   Thus the property $\omega_0\in L^p(\mathbb{T}^2)$, $p>1$, which is propagated for all later times, is particularly interesting because for $p\in(1,3/2]$ these spaces are Onsager-supercritical. Such a result is completely open in the presence of physical boundaries. 
\end{rem}
}
   
Following Bardos \& Titi \cite{BT18}, the strategy of proof that we employ for Theorems \ref{thm:onsager}, \ref{thm:onsagerNS} and \ref{thm:onsagerNS'} is to study the equations of motion for the coarse-grained (mollified) and {localized} velocity field $\eta_{\ell, h}(d(x))\ol{u}_\ell(x)$ where
\be\label{locFilt}
\ol{u}_\ell(x) = \int_\Omega G_{\ell}(r) u(x+r)dr,
\ee
 with $x\in \Omega^\ell$ and  $G_\ell(r):= \ell^{-d} G\big({r}/{\ell}\big)$. Here $\ell$ and $h$ are small numbers. The localization function $\eta_{\ell, h}(d(x))$ vanishes in  the strip $\Omega_{h-\ell}$ and is identically $1$ outside the strip $\Omega_h$. The scale $\ell$ is chosen sufficiently small compared to $h$  so that $\eta_{\ell, h}(d(x))\ol{u}_\ell(x)$ is defined all over $\Omega$.   The corresponding localized resolved energy balance for $\eta_{\ell, h}(d(x))\frac{1}{2} |\ol{u}_\ell|^2$ is then analyzed and scale-transfer terms in the bulk and at the boundary are estimated. For Theorem \ref{thm:onsager}, in order to make the Besov regularity assumption \eqref{bndassum1} only in the interior, the idea is to take the {\it  iterated limit $\ell\to 0$ taken first, and then subsequently $h\to 0$}. Details are given in \S \ref{thm1pf}.   For Theorems \ref{thm:onsagerNS} and  \ref{thm:onsagerNS'}, optimizing the relationship between $\nu$ and $\ell$ gives $\ell\sim h\sim \nu^\beta$, which sets the ``boundary layer" thickness.  Corresponding near-wall assumptions are made just within that thin layer, and interior Besov regularity is again assumed.

  \begin{rem}
An alternative strategy would be to analyze the equations  for a coarse-grained velocity field with a \emph{variable} length scale $\ell:=\ell(x)$ defined by the operation
\be\label{LESmol}
\ol{u}_{\ell(x)}(x) = \int_\Omega G_{\ell(x)}(r) u(x+r)dr.
\ee
Modeling and simulating the resulting equations of motion is  the basic approach of large-eddy simulation (LES) in the context of wall-bounded turbulence.  The length-scale $\ell:=\ell(x)$ now depends on $x$ in such a way that that $\ell(x)\to 0$ as ${\rm dist}(x,\partial \Omega)\to 0$.  
 This presents a complication for modeling which is not present on domains without boundary and is well known to the LES community; such filtering does not commute with spatial gradients.   Commutation errors produce additional terms which are unclosed and require modeling \cite{Fureby97,Ghosal99,Bos05,Geurts06}.   The decreasing filtering length--scale at the wall also makes the LES modeling method prohibitively expensive, since very fine grids must be employed near the boundary, see discussion in \cite{Chen12}.  We do not adopt this approach here as localized filtering \eqref{locFilt} yields sharper results.  
 \end{rem}

 In sections \ref{thm1pf}, \ref{NSthmProof}, \ref{thm:onsagerNS'} and  \ref{thm4pf} we provide proofs  of Theorems  \ref{thm:onsager},  \ref{thm:onsagerNS},  \ref{thm:onsagerNS'} and  \ref{thm:inviscid} respectively.

\section{Proof of Theorem \ref{thm:onsager}}\label{thm1pf}
For the sake of simplicity, we will proceed as if $u$ is differentiable in time. The extra arguments needed to mollify in time are straightforward, see e.g. \S 2 of \cite{DE17comp}.   In addition, because the Euler equations is spatially translation invariant, we may assume without loss of generality that the origin is an interior point of $\Omega$. 
Since $\partial \Omega$ is $C^2$ and compact, there exists an $h(\Omega)$ sufficiently small such that 
\begin{enumerate}[label=(\roman*)]
\item for any $x\in \ol{\Omega_{h(\Omega)}}$, the function $x\mapsto d(x)$ is $C^1(\ol{\Omega_{h(\Omega)}})$,
\item for any $x\in \ol{\Omega_{h(\Omega)}}$, there exists a unique point $\pi(x)\in \partial\Omega$ such that 
\be\label{distDeriv}
d(x) = |x-\pi(x)|, \qquad \nabla d(x)= -\hat{n}(\pi(x)).  
\ee
\end{enumerate} 
See e.g. Bardos \& Titi \cite{BT18}. In what follows we will denote
\[
\Omega^a:=\Omega\setminus \Omega_a.
\]
Let  $\ell$ and $h$ be small numbers such that 
\be\label{scalerange}
0<\ell<\frac{h}{4}<\frac{\ve}{8}<h(\Omega).
\ee
Let $G\in C_0^\infty(\Omega)$ be a nonnegative smooth radial function satisfying ${\supp}(G)\subset B_1(0)$ and $\int_{\Omega}G(r)dr=1$. We denote $G_\ell(r):=\ell^dG({r}/{\ell})$ and define the usual mollifier
\be\label{mollif}
\ol{f}_\ell(x):=\int_\Omega G_\ell(r)f(x+r)dr,\quad x\in \Omega^\ell.
\ee
To further localize near the boundary, we introduce  $\eta_{\ell, h}:\mathbb{R}\mapsto [0,1]$ a smooth non-decreasing function\footnote{The construction of such a function is standard and can be accomplished, for example, by mollifying the affine function
$$
 \eta_\ve(x) = (\varphi_\ve * \eta)(x), \qquad \eta(x) = \begin{cases} 0 & x<a\\
\frac{x-a}{b-a} & a\leq x<b\\
1& x\geq b
\end{cases},
$$
specifying $a,b$ such that $h-\ell<a<b<h$ and $\ve=b-a=\ell/4$ and setting $ \eta_{\ell, h}(x):=\eta_\ve(x)$.} which is $0$ on $\in(-\infty, h-\ell]$ and $1$ on $[h,\infty)$ and is such that $\| \eta_{\ell, h}'\Vert_{L^\infty(\Rr)}\le C\ell^{-1}$ for some constant $C$ independent of $\ell$ and $h$.  We denote $\theta_{\ell,h}(x):= \eta_{\ell, h}(d(x))$. See Figure \ref{fig}.

\begin{figure}[h!]
\centering
\includegraphics[width=0.35 \textwidth]{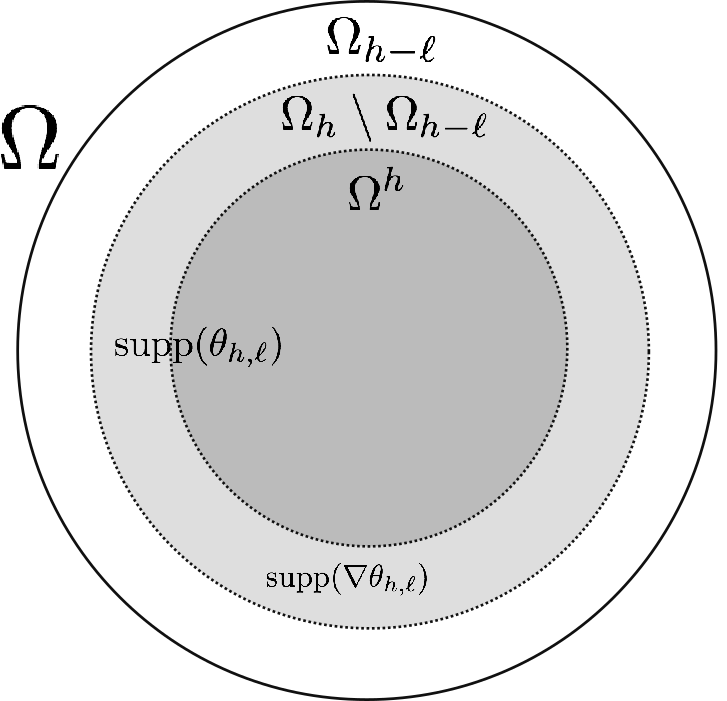}
\caption{
Localization region depicted.
}\label{fig}
\end{figure}

We choose test functions of the form $\varphi(r;x)= \theta_{\ell,h}(x) \ol{u}_\ell(x) G_{\ell}(r-x)$ in Eqn. \eqref{weakEuler}.  Note that such functions are infinitely differentiable and satisfy ${\rm supp}_r (\varphi)\subset \Omega$ due to the localization.  This is because $x\in \Omega^{h-\ell}$ on the support of $\theta_{\ell,h}(x)$, and also $|x-r|\leq \ell$ on the support of $G_{\ell}$, so that for all $r\in \Omega_{h(\Omega)}\cap{\rm supp}_r (\varphi)$
\[
d(r) = |r-\pi(r)|\geq |x-\pi(r)| - |x-r| \geq d(x) - \ell\geq h-2\ell>h/2
\]
using \eqref{scalerange}. Thus, $\varphi$ is supported away from the boundary in $r$.  Then, in the formula  Eqn. \eqref{weakEuler} the integration is over the variable $r$ and we see that
\begin{align}\nonumber
\theta_{\ell,h}(x) \ol{u}_\ell(x) \cdot \int_{\Omega-\{x\}}     \partial_tu(x+r)   G_{\ell}(r)   dr  - &\theta_{\ell,h}(x)   \ol{u}_\ell(x) \cdot \int_{\Omega-\{x\}}    (u\otimes u)(x+r)   \cdot \nabla_r G_{\ell}(r)   dr \\
&\qquad - \theta_{\ell,h}(x)   \ol{u}_\ell(x) \cdot   \int_{\Omega-\{x\}} p(x+r) \nabla_r G_{\ell}(r)  dr =0.
\end{align}
Provided ${\rm supp}(G_\ell)\subseteq B_\ell(0) \subset \Omega-\{x\}$, the domain of integration can be replaced by $\Omega$ in the above formula and we may write the expression in terms of mollifications.  This is the case when $x\in \Omega^{\ell}$.  The support of $\theta_{\ell,h}$ is contained in $\Omega^{h-\ell}\subset \Omega^{3\ell}\subset \Omega^\ell$ because we chose $h>4\ell$.   Thus, for all $x\in \Omega$, we have
\begin{align*}
\theta_{\ell,h}\Big[\partial_t \ol{u}_\ell \cdot \ol{u}_\ell   +    \ol{u}_\ell \cdot  \nabla \cdot   (\ol{u}_\ell\otimes \ol{u}_\ell) + \ol{u}_\ell \cdot \nabla\cdot \tau_\ell(u,u)  +  \ol{u}_\ell \cdot \nabla   \ol{p}_\ell\Big] =0
\end{align*}
where we have introduced the coarse-graining cumulant $\tau_\ell(u,u) = \ol{(u\otimes u)}_\ell - \ol{u}_\ell\otimes \ol{u}_\ell$ and used the fact that mollification commutes with spatial gradients.  After performing some straightforward calculations, we obtain the localized resolved kinetic energy balance
\begin{align}
\partial_t\left(\theta_{\ell,h} \frac{1}{2} |\ol{u}_\ell|^2\right) &+  \nabla \cdot \Big(\theta_{\ell,h} J_{\ell, h}[u] \Big)= \Pi_{\ell, h}[u] + B_{\ell, h}[u]
\end{align}
where we have introduced the \emph{spatial current} $J_{\ell, h}$ 
\be
J_{\ell, h}[u]  := \left(\frac{1}{2}|\ol{u}_\ell|^2 +\ol{p}_\ell \right) \ol{u}_\ell+ \ol{u}_\ell \cdot \tau_\ell(u,u)
\ee
as well as the \emph{bulk energy flux} $\Pi_{\ell, h}$ and the \emph{inertial boundary production} $B_{\ell, h}$ 
\begin{align}\label{piell}
&\Pi_{\ell, h}[u] = \theta_{\ell,h}  \nabla \ol{u}_\ell :\tau_\ell(u,u),\\ \label{Bell}
& B_{\ell, h}[u] = \nabla \theta_{\ell,h} \cdot J_{\ell, h}[u]. 
\end{align}
Integrating over space and time, we have the global resolved energy balance for all $t\in [0, T]$
\be\label{CGBalance}
\int_\Omega \theta_{\ell,h}(x) \frac{1}{2} |\ol{u}_\ell(t)|^2 dx = \int_\Omega \theta_{\ell,h}(x) \frac{1}{2} |\ol{u}_\ell(0)|^2 dx  +\int_0^t  \int_\Omega \Pi_{\ell, h}[u(x,t)] dx ds + \int_0^t  \int_{\Omega} B_{\ell, h}[u(x,t)]  dxds.
\ee
To obtain \eqref{CGBalance}, we used the divergence theorem and the fact that $\theta_{\ell,h}(x)= 0$ for $x\in \partial\Omega$ to deduce that the integral $ \int_\Omega \nabla \cdot \Big(\theta_{\ell,h} J_{\ell, h}[u] \Big)dx=0$.

We now develop appropriate estimates for the bulk flux and inertial boundary production terms which vanish under our regularity assumptions in the iterated limit $\ell\to 0$ first, and subsequently $h\to 0$.

\subsection{Estimate for the bulk energy flux $\Pi_{\ell, h}$}\label{bulkFluxEstimates}

First we remark that  $\Pi_{\ell, h}$ is supported on the same set as $\theta_{h,\ell}$, namely $\Omega^{h-\ell}$. We have therefore that
\be\label{piterms}
\int_0^t \| \Pi_{\ell, h}[u]\|_{L^1(\Omega)}ds= \int_0^t \| \Pi_{\ell, h}[u]\|_{L^1(\Omega^{h-\ell})}ds.
\ee
Since $0$ is an interior point of $\Omega$, $\int_\Omega (\nabla G)_\ell(r)dr=0$ provided that $\ell$ is sufficiently small. Mollified gradients of $f\in L^1(\Omega)$ can be expressed in terms of increments
\begin{align}\label{derivateEst}
\nabla \ol{f}_\ell(x) &= -\frac{1}{\ell}\int_{\Omega} (\nabla G)_{\ell}(r) \delta f(r;x) dr, \qquad \delta f(r;x):= f(x+r)-f(x)
\end{align}
for  $x\in \Omega^\ell$.  The  Constantin-E-Titi \cite{CET94} commutator identity shows the same is true for the cumulants 
\begin{align}\label{tauexp}
{\tau}_\ell(f,g) &=\int_{\Omega}  G_{\ell}(r)  \delta f(r;x)\otimes  \delta g(r;x) dr- \left(\int_{\Omega}  G_{\ell}(r)  \delta f(r;x)dr\right)\otimes\left(\int_{\Omega}  G_{\ell}(r)  \delta g(r;x)dr\right)
\end{align}
for all $f,g \in L^2(\Omega)$ and for all $x\in\Omega^\ell$.  
 We note that  $x+r$ belongs to $\Omega^{h/2}$ for all $x\in\Omega^{h-\ell}$ and $|r|<\ell<h/4$. Thus, both $x$ and $x+r$ lie in $\Omega^{h/2}\Subset \Omega$. { In view of \eqref{derivateEst} and \eqref{tauexp}, $\Pi_{\ell, h}$ is a sum of two terms which can be treated in the same way. Let us consider just first term 
  \[
I=-\frac{1}{\ell}\theta_{\ell, h}\int_0^t\int_{\Omega^{h-\ell}}\left\{\int_{\Omega} (\nabla G)_{\ell}(r)\delta u(r;x) dr:\int_{\Omega}  G_{\ell}(r)  \delta u(r;x)\otimes  \delta u(r;x) dr\right\}dxds.
 \]
 This contribution can be estimated as follows
 \begin{align*}
 I&\le \frac{1}{\ell}  \left\|\int_{\Omega} (\nabla G)_{\ell}(r) \delta u(r;x) dr \right\|_{L^3_{x, s}(\Omega^{h-\ell}\times (0, t))}\left\|\int_{\Omega}  G_{\ell}(r)  \delta u(r;x)\otimes  \delta u(r;x) dr \right\|_{L^{3/2}_{x,s}(\Omega^{h-\ell}\times (0, t))}\\
 &\le \frac{1}{\ell} \int_{\Omega} |(\nabla G)_{\ell}(r)||r|^{1/3}\frac{\| \delta u(r;\cdot) \|_{L^3(0, t; L^3(\Omega^{h-\ell}))}}{|r|^{1/3}}dr\int_{\Omega}  G_{\ell}(r) |r|^{2/3}\frac{\| \delta u(r;\cdot) \|_{L^3(0, t; L^3(\Omega^{h-\ell}))}^2}{|r|^{2/3}}dr\\
  &\le \Big(\sup_{|r|<\ell}\frac{\| \delta u(r;\cdot) \|_{L^3(0, t; L^3(\Omega^{h-\ell}))}}{|r|^{1/3}}\Big)^3\int_{\mathbb{R}^d} |\rho|^{1/3} |\nabla G(\rho)|d\rho\int_{\mathbb{R}^d}  |\rho|^{2/3}G(\rho)d\rho\\
  &\le C\Big(\sup_{|r|<\ell}\frac{\| \delta u(r;\cdot) \|_{L^3(0, t; L^3(\Omega^{h-\ell}))}}{|r|^{1/3}}\Big)^3
 \end{align*}
 where we used H\"older's inequality and $G$ is smooth and compactly supported. Using condition \ref{bndassum1} we deduce that $\lim_{\ell\to 0}I=0$. The other term follows by an analogous computation.} Therefore, for any $h$ fixed,
\be\label{limFlux}
\lim_{\ell\to 0}  \int_0^t\!  \| \Pi_{\ell, h}[u(s)] \|_{L^1(\Omega)} ds =0 \qquad\forall t\in (0, T).
\ee
\subsection{Estimate for the inertial boundary production $B_{\ell, h}$} \label{boundaryEstEuler}
Note that $B_{\ell, h}[u]$ can be expressed as
\begin{align}
B_{\ell, h}[u]  &= -  \eta_{\ell,h}'(d(x)) \left( \left(\frac{1}{2}|\ol{u}_\ell|^2 +\ol{p}_\ell \right)\hat{n}(\pi(x))\cdot \ol{u}_\ell+ \ol{u}_\ell \cdot \tau_\ell(u,u)\cdot \hat{n}(\pi(x))\right) \label{Bform}
\end{align}
and has the same support as $\eta_{\ell,h}'$, namely $\Omega_{h}\setminus \Omega_{h-\ell}\subset \Omega_\ve$.  
We write
 \begin{align}\nonumber
 \hat{n}(\pi(x))\cdot \ol{u}_\ell(x)&= \int_\Omega G_{\ell}(r)\ [\hat{n}(\pi(x))- \hat{n}(\pi(x+r))]\cdot u(x+r)  dr + \int_\Omega G_{\ell}(r)\ \hat{n}(\pi(x+r))\cdot u(x+r)dr.
  \end{align}  
Since   $\hat{n}\circ \pi$ is continuous in $\ol{\Omega_\ve}$, for any $\delta>0$ there exists $\rho=\rho(\delta)>0$ such that
\be \label{Bbd1}
 |\hat{n}(\pi(x))- \hat{n}(\pi(x+r))|\le \delta
\ee
 for all $x\in \Omega_{h}\setminus \Omega_{h-\ell}$ and $|r|<\ell<\rho$. It follows that
  \be \label{Bbd2}
  |\hat{n}(\pi(x))\cdot \ol{u}_\ell(x,t)| \lesssim \delta \|u(t)\|_{L^\infty(\Omega_{\ve})}+ \sup_{x\in\Omega_{2h}} |\hat{n}(\pi(x))\cdot u(x,t)|.
 \ee
 Recalling the definition of $\tau_\ell(u, u)$ we obtain similarly that
 \be  \label{Bbd3}
 | \tau_\ell(u,u)\cdot \hat{n}(\pi(x))| \lesssim \Big(\delta \|u(t)\|_{L^\infty(\Omega_{\ve})}+ \sup_{x\in\Omega_{2h}} |\hat{n}(\pi(x))\cdot u(x,t)|\Big) \|u(t)\|_{L^\infty(\Omega_{\ve})}.
 \ee
Using these bounds, together with the facts that $\Vert \eta'_{\ell, h}(d(x))\Vert_{L^\infty}\le C\ell^{-1}$ and $|\Omega_h\setminus\Omega_{h-\ell}|\le C\ell$, we obtain
\begin{align}\nonumber
\int_0^t  \|B_{\ell, h}[u(s)]\|_{L^1(\Omega)} ds&\lesssim \|u\|_{L^3(0,T;L^\infty(\Omega_{\ve}))}^2 \|\sup_{x\in\Omega_{2h}} |\hat{n}(\pi(x))\cdot u(x,\cdot)|\|_{L^3((0,T))} \\
&\quad +  \delta\Big(\|u\|_{L^3(0,T;L^\infty(\Omega_{\ve}))}^3+  \|u\|_{L^3(0,T;L^\infty(\Omega_{\ve}))} \|p\|_{L^{3/2}(0,T;L^\infty(\Omega_{\ve}))}\Big)  \label{Bbd4}
\end{align}
provided $\ell<\rho$. Thus, by assumptions \eqref{bndassum1} and \eqref{bndassum3}, the boundary production has vanishing contribution
\be\label{limBP}
\lim_{h\to 0} \lim_{\ell\to 0}\int_0^t  \|B_{\ell, h}[u(s)]\|_{L^1(\Omega)} ds =0.
\ee
\subsection{Convergence of resolved energy}\label{resolvedEnergyConv}
In view of \eqref{CGBalance}, \eqref{limFlux}, and \eqref{limBP}, we have for all $t\in [0,T]$
\be\label{resEcons}
\lim_{h\to 0}\lim_{\ell\to 0}\left| \int_\Omega \theta_{\ell,h}(x) \frac{1}{2} |\ol{u}_\ell(t)|^2 dx -\int_\Omega \theta_{\ell,h}(x) \frac{1}{2} |\ol{u}_\ell(0)|^2 dx  \right| =0.
\ee
Combining \eqref{resEcons} with the following lemma in the case of $p=2$, we conclude that $u$ conserves energy. 
\begin{lemma}
For any $f\in L^p(\Omega)$, $p\in [1, \infty]$, we have
\be
\lim_{h\to 0}\lim_{\ell\to 0}\Vert  \theta_{\ell,h}  \overline f_\ell-f\Vert_{L^p(\Omega)}=0.
\ee
\end{lemma}
\begin{proof}
Fix a small number $h>0$ and consider $0<\ell<h/4$. Since $f\in L^p(\Omega)$ we know that $\ol{f}_\ell\to f$ in $L^p_{loc}(\Omega)$ as $\ell\to 0$ (see e.g. Appendix C.5 \cite{Evans10}). In particular, $\ol{f}_\ell\to f$ in $L^p(\Omega^h)$  as $\ell\to 0$. Next, 
\begin{align*}
\Vert  \theta_{\ell,h} \overline f_\ell-f\Vert_{L^p(\Omega_h)}&=\Vert  \theta_{\ell,h}  \overline f_\ell-f\Vert_{L^p(\Omega_h\setminus \Omega_{h-\ell})}+\Vert f\Vert_{L^p(\Omega_{h-\ell})}\\ 
&\leq \Vert (\ol f_\ell-f) \theta_{\ell,h} \Vert_{L^p(\Omega_h\setminus \Omega_{3h/4})}+ \Vert ( \theta_{\ell,h} -1)f \Vert_{L^p(\Omega_h\setminus \Omega_{3h/4})}+\Vert  f\Vert_{L^p(\Omega_{3h/4})}\\
&\leq \Vert \ol f_\ell-f\Vert_{L^p(\Omega_h\setminus \Omega_{3h/4})}+2\Vert  f\Vert_{L^p(\Omega_{h})}
\end{align*}
where we used the fact that  $\Omega_h\setminus \Omega_{h-\ell}\subset \Omega_h\setminus \Omega_{3h/4}$. Since $h$ is fixed, the closure of $\Omega_h\setminus \Omega_{3h/4}$ is contained in $\Omega$. Hence, $\Vert  \ol f_\ell-f\Vert_{L^p(\Omega_h\setminus \Omega_{3h/4})}\to 0$ as $\ell\to 0$. Finally, letting $h\to 0$ we have $\Vert  f\Vert_{L^p(\Omega_{h})}\to 0$ by Lebesgue dominated convergence theorem.
\end{proof}

\section{Proof of Theorem \ref{thm:onsagerNS}}\label{NSthmProof}

 We take the same approach as for the proof of Theorem \ref{thm:onsager}.   A priori, Leray-Hopf solutions are known to satisfy the weak form \eqref{weakNS} with divergence-free test fields, however the localized-mollified velocity we wish to test against to obtain the resolved energy balance is not solenoidal.    We must therefore appeal to the following Theorem of \cite{Jmm14} (see also \cite{Giga91,Sohr86} for corresponding results with smoother initial data):
 
\begin{thm}[\protect{see \cite[Theorem 1]{Jmm14}}]\label{reintro:p}
 Let $\Omega\subset \mathbb{R}^d$ be a bounded domain with  $C^2$ boundary, $\partial \Omega$ and $u^\nu$ be a Leray-Hopf weak solution to the Navier-Stokes equations \eqref{NSeq}--\eqref{NSbc} with initial data $u_0^\nu\in H(\Omega)$. There exists a pressure $p^\nu\in L^{r_0}(0, T; W^{1, s_0}(\Omega))$ for some $r_0,~s_0\in (1, \infty)$ depending only on  $d$, such that 
   \be\label{weakNS2}
 \int_0^T\int_\Omega  \big(u^\nu \cdot  \partial_t\varphi   +  u^\nu \otimes u^\nu  :  \nabla  \varphi  +p^\nu \nabla \cdot  \varphi  +\nu u^\nu\cdot \Delta \varphi)dxdt =0
\ee
for all $\varphi\in C_0^\infty(\Omega\times (0, T))$.
\end{thm}

Thus, we may choose smooth test functions $\varphi(r;x)= \theta_{\ell,h}(x) \ol{u}_\ell(x) G_{\ell}(r-x)$ in Eqn. \eqref{weakNS2}.   As a result, we obtain the localized resolved kinetic energy balance
\begin{align}\label{ebal1}
\partial_t\left(\theta_{\ell,h} \frac{1}{2} |\ol{(u^\nu)}_\ell|^2\right) &+  \nabla \cdot \Big(\theta_{\ell,h} J_{\ell, h}^\nu[u^\nu] \Big)= \Pi_{\ell, h}[u^\nu] + B_{\ell, h}^\nu[u^\nu] - D_{\ell,h}^\nu[u^\nu]
\end{align}
with $\Pi_{\ell,h}$ defined by \eqref{piell}, $B_{\ell,h}^\nu$ is defined as in \eqref{Bell} but with the space transfer term replaced by 
\be
J_{\ell, h}^\nu[u^\nu] :=  \left(\frac{1}{2}|\ol{(u^\nu)}_\ell|^2 +\ol{(p^\nu)}_\ell \right) \ol{(u^\nu)}_\ell+ \ol{(u^\nu)}_\ell \cdot \tau_\ell(u^\nu,u^\nu) - \nu \nabla \ol{(u^\nu)}_\ell\cdot \ol{(u^\nu)}_\ell
\ee
and a new \emph{resolved dissipation} term
\be
 D_{\ell,h}^\nu[u^\nu]:= \nu \theta_{h,\ell} |\nabla \ol{(u^\nu)}_\ell|^2.
 \ee
Since $\theta_{h,\ell}|_{\partial \Omega}=0$,  the spatial divergence term in \eqref{ebal1} vanishes upon integration over $\Omega$ and we have
\begin{align}\nonumber
\int_\Omega\theta_{\ell,h}(x) &\frac{1}{2} |\ol{(u^\nu)}_\ell(t)|^2 dx = \int_\Omega\theta_{\ell,h}(x) \frac{1}{2} |\ol{(u^\nu)}_\ell(0)|^2 dx  \\
&\quad+\int_0^t  \int_\Omega \Pi_{\ell,h}[u^\nu(s)] dx ds + \int_0^t  \int_{\Omega_{h}} B_{\ell,h}^\nu[u^\nu(s)]dx ds -\int_0^t  \int_\Omega D_{\ell,h}^\nu[u^\nu(s)] dx ds
\end{align}
for all $t\in [0, T]$.  
We set the scales $\ell$ and $h$ to vary with viscosity according to
\be\label{scales}
 h=6\ell=\frac{3}{4} \nu^\beta, \qquad \red{\beta=\min\left\{1,\frac{1}{2(1-\sigma)}\right\}}.
\ee 
We now show that the contributions of $ D_{\ell,h}^\nu[u^\nu]$, $\Pi_{\ell,h}[u^\nu]$ and $B_{\ell,h}^\nu[u^\nu]$ vanish in the limit $\nu,\ell,h\to 0$ taken in the order set by \eqref{scales}.  We start by bounding the resolved dissipation.

\subsection{Estimate for the resolved dissipation $D_{\ell, h}^\nu$}\label{resolvDiss}

First we remark that  $D_{\ell, h}^\nu$ is supported in $\Omega^{h-\ell}$, hence
\[
\int_0^t \| D_{\ell, h}^\nu[u^\nu]\|_{L^1(\Omega)}ds= \int_0^t \| D_{\ell, h}^\nu[u^\nu]\|_{L^1(\Omega^{h-\ell})}ds.
\]
The mollified derivatives are then bounded using the formula \eqref{derivateEst} in terms of the assumed Besov regularity. In particular, note that that  $x+r$ belongs to $\Omega^{h/2}$ for all $x\in\Omega^{h-\ell}$ and $|r|<\ell<h/4$. Thus, both $x$ and $x+r$ lie in $\Omega^{2h/3}\subseteq \Omega^{\nu^\beta/2}\Subset \Omega$.  By assumption \eqref{bndassum1NS}, $u^\nu$ is uniformly bounded in $L^3(0,T; B_3^{\sigma,\red{c_0}}(\Omega^{\nu^\beta/2}))$ with \red{$\sigma\geq 1/3$}. We now estimate
\red{
 \begin{align*}
\int_0^t \| D_{\ell, h}^\nu[u^\nu]\|_{L^1(\Omega)}ds &\le \frac{\nu}{\ell^2}\left\| \int_{\Omega}  (\nabla G)_{\ell}(r) \delta u(r;x) dr\right\|_{L^3(0, t; L^3_x(\Omega^{h-\ell}))}\left\| \int_{\Omega}  (\nabla G)_{\ell}(r) \delta u(r;x) dr\right\|_{L^3(0, t; L^3_x(\Omega^{h-\ell}))}\\
&\le \frac{\nu}{\ell^2}\left( \int_{\Omega}  |(\nabla G)_{\ell}(r)||r|^{\sigma}\frac{\| \delta u(r;\cdot) \|_{L^3(0, t; L^3(\Omega^{h-\ell}))}}{|r|^{\sigma }}dr\right)^2\\
  &\le \nu\ell^{2(\sigma-1)}\Big(\sup_{|r|<\ell}\frac{\| \delta u(r;\cdot) \|_{L^3(0, t; L^3(\Omega^{h-\ell}))}}{|r|^{\sigma }}\Big)^2\left(\int_{\mathbb{R}^d} |\rho|^{\sigma }  |\nabla G(\rho)| d\rho\right)^2\\
  &\le C  \nu^{1+2\beta(\sigma-1)} \Big(\sup_{|r|<\ell}\frac{\| \delta u(r;\cdot) \|_{L^3(0, t; L^3(\Omega^{h-\ell}))}}{|r|^{\sigma}}\Big)^2.
 \end{align*}
 Since $1+ 2\beta(\sigma-1)\geq 0$ by our assumption on $\beta$, we deduce from the assumption \eqref{bndassum1NS} that
 \begin{align}\label{Dvan}
 \lim_{\nu\to 0}  \int_0^t  \int_{\Omega} D_{\ell,h}[u^\nu(s)]dx ds\Bigg|_{\ell,h\sim\nu^\beta} =0.
\end{align}
 }

\subsection{Estimate for the bulk energy flux $\Pi_{\ell, h}$}   

Following the argument in \S \ref{bulkFluxEstimates}, \ref{resolvDiss}, we have estimate
\red{
$$
\int_0^t \| \Pi_{\ell, h}[u^\nu]\|_{L^1(\Omega)}ds\leq C\ell^{3\sigma-1}\Big(\sup_{|r|<\ell}\frac{\| \delta u(r;x) \|_{L^3(0, t; L^3(\Omega^{h-\ell}))}}{|r|^{\sigma}}\Big)^3.
$$
 Since $\sigma\geq 1/3$, the uniform interior Besov regularity assumption \eqref{bndassum1NS} combined with the above yields
\begin{align}\label{Pivan}
 \lim_{\nu\to 0}  \int_0^t  \int_{\Omega} \Pi_{\ell,h}[u^\nu(s)]dx ds\Bigg|_{\ell,h\sim\nu^\beta} =0.
\end{align}
}

\subsection{Estimate for the inertial boundary production $B_{\ell, h}$}
Following the argument in \S \ref{boundaryEstEuler} and using uniform continuity and boundedness \eqref{bndassum2NS}, \eqref{bndassum3NS}, we obtain an identical bound  \eqref{Bbd4} with $\Omega_\ve$ replaced by $\Omega_{\nu^\beta}$.   Since our continuity-at-the-wall assumption \eqref{bndassum4NS} is stronger than the corresponding Euler assumption \eqref{bndassum3} which is for the normal component alone, these upper bounds vanish as $\nu\to 0$.  We must finally treat the additional term arising due to viscosity.  We proceed as follows
\be\label{B:nuterm}
\begin{aligned} 
\nu\nabla \theta_{h,\ell}\cdot \nabla \ol{(u^\nu)}_\ell \cdot \ol{(u^\nu)}_\ell &=-\nu \eta_{\ell,h}'(d(x)) \hat{n}(\pi(x))\cdot \nabla \ol{(u^\nu)}_\ell\cdot \ol{(u^\nu)}_\ell\\
&=\frac{\nu}{\ell} \eta_{\ell,h}'(d(x)) \int_\Omega (\nabla G)_\ell(r)\cdot \hat{n}(\pi(x))  \ u^\nu(x+r) \cdot \ol{(u^\nu)}_\ell(x) dr.
\end{aligned}
\ee
We remark the assumption of wall-continuity \eqref{bndassum4NS} for Navier-Stokes, instead of the weaker condition of wall-normal continuity \eqref{bndassum3} which we had for Euler, seems necessary to control this term as $\nu\to 0$.
Now, using the fact that $\|\eta_{\ell,h}'(d(x)) \|_{L^\infty(\Omega)} \leq C \ell^{-1}$ and $|\Omega_h\setminus \Omega_{h-\ell}|\leq C_\Omega \ell$, we have
\begin{align} 
\int_0^t\int_{\Omega_h\setminus \Omega_{h-\ell}} |\nu\nabla \theta_{h,\ell}\cdot \nabla \ol{(u^\nu)}_\ell \cdot \ol{(u^\nu)}_\ell|dx dt &\lesssim \frac{\nu}{\ell} \|u^\nu\|_{L^2(0,T;L^\infty(\Omega_{\nu^\beta}))} \|\sup_{x\in \Omega_{\nu^\beta}} |u^\nu(x,\cdot)|\|_{L^2((0,T))}
\end{align}
where we used the fact that $x+r\in \Omega_{\nu^\beta}$ for all $x\in \Omega_h$ and $r\in B_\ell(0)$.
Using the assumption \eqref{bndassum4NS},  together with the fact that $\nu/\ell = O(1)$ by the choice \eqref{scales},
this implies that
\begin{align}\label{Bvan}
 \lim_{\nu\to 0}  \int_0^t  \int_{\Omega} B_{\ell,h}^\nu[u^\nu(s)]dx ds\Bigg|_{\ell,h\sim\nu^\beta} =0.
\end{align}

\subsection{Vanishing viscous dissipation}\label{secEnCons}
Combining \eqref{Dvan},\eqref{Pivan}, and \eqref{Bvan}, we have 
\be\label{resolvedEConv}
\lim_{\nu\to 0} \left| \frac{1}{2} \int_{\Omega}\theta_{\ell,h}(x) |\ol{(u^\nu)}_\ell(x,0)|^2 dx-  \frac{1}{2} \int_{\Omega} \theta_{\ell,h}(x) |\ol{(u^\nu)}_\ell(x,t)|^2 dx \right|\Bigg|_{\ell,h\sim\nu^\beta}=0
\ee
 for all $t\in [0,T]$. Now we note that
 \be
 \int_{\Omega}\theta_{h,\ell} \ol{(|u^\nu(x,t)|^2)}_\ell dx  \leq        \int_{\Omega}\int_\Omega G_\ell(r) \theta_{h,\ell}(x) |u^\nu(x+r,t)|^2 drdx    \leq      \int_{\Omega} |u^\nu(x,t)|^2 dx.
    \ee
 This combined with the fact that energy of $u^\nu$ is non-increasing yields
\begin{align*} \nonumber
 0&\le \int_{\Omega} |u_0^\nu(x)|^2 dx-\int_{\Omega}|u^\nu(x, t)|^2 dx  \\
 &\leq  \int_{\Omega} |u_0^\nu(x)|^2 dx-\int_{\Omega}\theta_{h,\ell} \ol{(|u^\nu(x,t)|^2)}_\ell dx  \\
 &=\int_{\Omega} |u_0^\nu(x)|^2 dx-\int_{\Omega} \theta_{\ell,h}(x) |\ol{(u^\nu)}_\ell(x,t)|^2 dx-\int_{\Omega}\theta_{h,\ell} \tau(u^\nu(t), u^\nu(t))dx\\
 &\le \int_{\Omega} |u_0^\nu(x)|^2 dx-\int_{\Omega} \theta_{\ell,h}(x) |\ol{(u^\nu)}_\ell(x,t)|^2 dx\\
 &= \int_{\Omega} |u_0^\nu(x)|^2 dx-\int_{\Omega} \theta_{\ell,h}(x) |\ol{(u_0^\nu)}_\ell(x)|^2 dx+\int_{\Omega} \theta_{\ell,h}(x) |\ol{(u_0^\nu)}_\ell(x)|^2 dx-\int_{\Omega} \theta_{\ell,h}(x) |\ol{(u^\nu)}_\ell(x,t)|^2 dx
\end{align*}
where in the fourth line we used Jensen's inequality  to have $ \tau_\ell(f,f)\geq 0$ pointwise. Let $\nu_n$ be an {\it arbitrary} sequence of viscosity converging to $0$. Choose $\ell=\ell_n:=\ell(\nu_n)$ and $h=h_n:=h(\nu_n)$ defined by \eqref{scales}. By \eqref{resolvedEConv} the second difference on the right hand side vanishes in the limit $ \nu_n\to 0$. We claim that 
\be\label{claim:conv}
\lim_{n\to \infty} \Big[\int_{\Omega} |u_0^{\nu_n}(x)|^2 dx-\int_{\Omega} \theta_{\ell,h}(x) |\ol{(u_0^{\nu_n})}_\ell(x)|^2 dx\Big]=0.
\ee
With this granted, 
\be
\lim_{n\to \infty}  \Big[\int_{\Omega}|u^{\nu_n}(x,t)|^2 dx-\int_{\Omega} |u_0^{\nu_n}(x)|^2 dx\Big]=0,
\ee
and thus from the energy inequality for Leray-Hopf weak solutions, the viscous dissipation vanishes in the limit $\nu_n\to 0$, i.e.
\be
\lim_{n\to \infty}  {\nu_n}\int_0^t\int_\Omega |\nabla u^{\nu_n}(x, s)|dxds=0.
\ee
Since the sequence $\nu_n$ is arbitrary, we conclude no anomalous dissipation \eqref{noanomaly}.

It remains to prove the claim \eqref{claim:conv} which, by the assumption $u^\nu_0\to u_0$ in $L^2(\Omega)$, is equivalent to 
\be\label{claim:conv2}
\lim_{n\to \infty} \Big[\int_{\Omega} |u_0(x)|^2 dx-\int_{\Omega} \theta_{\ell,h}(x) |\ol{(u_0^{\nu_n})}_\ell(x)|^2 dx\Big]=0.
\ee
To this end, we write 
\begin{align} \nonumber
  \ol{(u_0^{\nu_n})}_\ell(x)-u_0(x)&=\int_\Omega G_\ell(r)[u_0^{\nu_n}(x+r)-u_0(x)]dr\\
  &=\int_\Omega G_\ell(r)\Big([u_0^{\nu_n}(x+r)-u_0(x+r)]+[u_0(x+r)-u_0(x)]\Big)dr.
  \end{align}
  Then, for any $U\Subset V\Subset \Omega$, we have
  \begin{align}\nonumber
  \|  \ol{(u_0^{\nu_n})}_\ell-u_0\|_{L^2(U)} &\le \int_\Omega G_\ell(r)\| u^{\nu_n}_0(\cdot+r)-u_0(\cdot+r)\|_{L^2(U)}dr+\int_\Omega G_\ell(r)\| u_0(\cdot+r)-u_0(\cdot)\|_{L^2(U)}dr\\
  &\le \| u_0^{\nu_n}-u_0\|_{L^2(V)}+\sup_{|r|<\ell}\| u_0(\cdot+r)-u_0(\cdot)\|_{L^2(U)}
  \end{align}
  where we used the fact that $U-\{r\}\subset V$ for $|r|<\ell\sim \nu_n^\beta$ sufficiently small. This implies that 
\be\label{converge:L2}
\lim_{n\to \infty} \|  \ol{(u_0^{\nu_n})}_\ell-u_0\|_{L^2(U)} =0.
\ee
Consequently
\be\label{conv:L2:2}
\lim_{n\to \infty}\int_U\theta_{\ell,h}(x) |\ol{(u_0^{\nu_n})}_\ell(x)|^2 dx= \int_{U}|u_0(x)|^2 dx.
\ee
 For $\delta$ sufficiently small and $\ell<\delta$ we have 
\be
\int_{\Omega_\delta}\theta_{\ell,h}(x) |\ol{(u_0^{\nu_n})}_\ell(x)|^2 dx\le \int_{\Omega_{2\delta}} |u_0^{\nu_n}(x)|^2dx,
\ee
hence, from our assumption that $u_0^\nu\to u_0$ strongly in $L^2(\Omega)$,
\be\label{delta1}
\lim_{\delta\to 0}\lim_{n\to \infty}\int_{\Omega_\delta}\theta_{\ell,h}(x) |\ol{(u_0^{\nu_n})}_\ell(x)|^2 dx\le \lim_{\delta\to 0}\int_{\Omega_{2\delta}} |u_0(x)|^2dx=0.
\ee
On the other hand, it follows from \eqref{conv:L2:2} with $U=\Omega^\delta$ that
\be\label{delta2}
\lim_{\delta\to 0}\lim_{n\to \infty}\int_{\Omega^\delta}\theta_{\ell,h}(x) |\ol{(u_0^{\nu_n})}_\ell(x)|^2 dx= \int_{\Omega}|u_0(x)|^2 dx.
\ee
 Combining \eqref{delta1} and \eqref{delta2}, we have proved that 
\be
\lim_{n\to \infty}\int_{\Omega}\theta_{\ell,h}(x) |\ol{(u_0^{\nu_n})}_\ell(x)|^2 dx=\int_\Omega |u_0(x)|^2dx.
\ee
This concludes the proof of \eqref{claim:conv2}.
\subsection{Convergence to weak Euler}\label{SecWeakE}
By assumptions \eqref{bndassum1NS} and \eqref{bndassum1NS}, $u^\nu$ is uniformly in $\nu$ bounded in $L^3(\Omega\times (0, T))$. Thus, we have weak convergence
  \be\label{conv:wL3}
  u^{\nu_k}\rightharpoonup u\quad\text{in} ~L^3(\Omega\times (0, T))
  \ee
for some $u$ and some subsequence of $\nu_n\to 0$. Fix two arbitrary open sets $U$ and $V$ with $U\Subset V\Subset \Omega$. By means of Aubin-Lions's lemma one can prove the following 
  \begin{lemma}\label{lemm:inviscid}
  If  $u^{\nu_k}$ are uniformly in $k$ bounded in $L^3(0, T; B^{s, \infty}_3(V))$ with $s>0$ and $\nabla p^{\nu_k}$ are uniformly in $k$ bounded in $L^{3/2}(0, T; W^{-1, 3/2}(V)$ then there exist a subsequence $u^{\nu_{k_j}}$ and a vector field $u^\infty$, both depending on $U$ a priori, such that $u^{\nu_{k_j}}\to u^\infty$ in $L^3(U\times (0, T))$. Moreover, $u^\infty$ is a weak Euler solution on $U\times (0, T)$.
  \end{lemma}
 \begin{proof} From the assumptions we have that $\partial_t u^{\nu_k}$ are uniformly in $k$ bounded in $L^{3/2}(0, T; W^{-1, 3/2}(V))$. This is because $\sqrt {\nu_k}\Delta u^{\nu_k}$ are uniformly in $k$ bounded in $L^2(0, T; H^{-1}(V))\subset L^{3/2}(0, T; W^{-1, 3/2}(V))$ which follows from energy inequality of Leray-Hopf solutions. Note that  $B^{s, \infty}_3(V)\subset L^3(U)\subset W^{-1, 3/2}(U)$ with the first embedding being compact and the second one being continuous. By virtue of Aubin-Lion's lemma there exist $u^\infty\in L^3(0, T; L^3(U))$ and a subsequence $\nu_{k_j}$ such that 
\[
u^{\nu_{k_j}}\to u^\infty\quad\text{in}~L^3(0, T; L^3(U)).
\]
We can now pass to limit $j\to \infty$ in the weak formulation of Navier-Stokes equations and obtain that $u^\infty$ is a weak Euler solution on $U\times (0, T)$.
\end{proof}

  Take $\nu_k$ sufficiently small such that $V\subset \Omega^{\nu_k^{\beta}/2}$.
  Note that from assumption \eqref{bndassum1NS}, we have that $u^\nu$ is uniformly in $\nu$ bounded in $L^3(0, T; B^{\sigma, \infty}_3(V))$ with $\sigma>0$. By virtue of Theorem \ref{reintro:p} and the trace theorem, the pressure trace $p^\nu\vert_{\partial \Omega}$ is well defined in $L^{s_0}(\partial\Omega)$ for a.e. $t\in [0,T]$.  It then follows from the  uniform boundedness assumption \eqref{bndassum3NS} that $p^\nu\vert_{\partial\Omega} \in L^\infty(\partial\Omega)$. In other words, $p^\nu$ satisfies the Dirichlet problem
   \begin{align}
\qquad -\Delta p^\nu &= \nabla \otimes \nabla : (u^\nu\otimes u^\nu), \qquad \text{in} \ \Omega,\\
p^\nu &\in L^\infty(\partial\Omega), \qquad \quad \quad \qquad \text{on} \ \partial\Omega.
\end{align}
   We deduce by the method of transposition (see Remark \ref{rema:pressure}) that $p^\nu\in L^{3/2}(0, T; L^{3/2}(\Omega))$ uniformly in $\nu$. Thus,  $\nabla p^{\nu_k}$ are uniformly in $k$ bounded in $L^{3/2}(0, T; W^{-1, 3/2}(V)$.  Then, by Lemma \ref{lemm:inviscid}, every subsequence of $u^{\nu_{k}}$ has a subsequence converging to some $u^\infty$ in $L^3(U\times (0, T))$, and  $u^\infty$ is a weak Euler solution on  $U\times (0, T)$. In view of \eqref{conv:wL3}, necessarily $u^\infty\equiv u$ on $U\times (0, T)$. Therefore, the whole sequence $u^{\nu_k}$ converges to  $u$ in $L^3(U\times (0, T))$. Moreover, since $U$ is arbitrary we deduce that $u$ is  a weak Euler solution on  $\Omega\times (0, T)$.

\section{Proof of Theorem \ref{thm:onsagerNS'}}

We will follow the proof of Theorem \ref{thm:onsagerNS}. Under assumptions \eqref{bndassum1NS'}, \eqref{bndassum2NS'} and \eqref{bndassum3NS'}, estimates for the resolved dissipation $D_{\ell, h}^\nu$ and the bulk energy flux $\Pi_{\ell, h}$ are identical to those in the proof of Theorem \ref{thm:onsagerNS}. Thus, it remains to treat the inertial boundary production 
\[
 B^\nu_{\ell, h}[u^\nu]=\nabla \theta_{\ell,h} \cdot J_{\ell, h}[u^\nu],\qquad J_{\ell, h}^\nu[u^\nu] =  \left(\frac{1}{2}|\ol{(u^\nu)}_\ell|^2 +\ol{(p^\nu)}_\ell \right) \ol{(u^\nu)}_\ell+ \ol{(u^\nu)}_\ell \cdot \tau_\ell(u^\nu,u^\nu) - \nu \nabla \ol{(u^\nu)}_\ell\cdot \ol{(u^\nu)}_\ell.
\]
First, we have
\begin{align*}
\int_0^t \int_\Omega |\nu\nabla \theta_{h,\ell}\cdot \nabla \ol{(u^\nu)}_\ell \cdot \ol{(u^\nu)}_\ell |dxds&\lesssim \frac{\nu}{\ell} \int_0^T\int_{\Omega_h\setminus\Omega_{h-\ell}}|\ol{(\nabla u^\nu)}_\ell \cdot \ol{(u^\nu)}_\ell |dxds\\
&\lesssim \frac{\nu}{\ell} \| \nabla u^\nu\|_{L^2(0, T; L^2(\Omega_{\nu^\beta}))}\| u^\nu\|_{L^2(0, T; L^2(\Omega_{\nu^\beta})}\\
&\lesssim \nu \| \nabla u^\nu\|_{L^2(0, T; L^2(\Omega_{\nu^\beta}))}^2=o_\nu(1)
\end{align*}
where  in the last line  we used the choice \eqref{scales} and assumption \eqref{cd:vanish}, and Poincar\'e's inequality with $u|_{\partial\Omega}=0$ to conclude
\be\label{Poincare:strip}
\| u^\nu\|_{L^2(0, T; L^2(\Omega_{\nu^\beta})}\le C\nu^\beta\| \nabla u^\nu\|_{L^2(0, T; L^2(\Omega_{\nu^\beta}))}.
\ee
Next, we consider the term
\begin{align*}
\int_0^t\int_\Omega|\nabla \theta_{\ell,h} \cdot \big(|\ol{(u^\nu)}_\ell|^2 \ol{(u^\nu)}_\ell\big)|dxds&\leq \int_0^t\int_{\Omega_h\setminus\Omega_{h-\ell}}|\eta_{\ell,h}'(d(x))| |\ol{(u^\nu)}_\ell|^2 |\hat n(\pi(x))\cdot \ol{(u^\nu)}_\ell|dxds\\
&\lesssim \frac{1}{\ell}\|u^\nu\|_{L^\infty(0, T; L^\infty(\Omega_{\nu^\beta}))}\Vert {u^\nu}\|^2_{L^2(0, T; L^2(\Omega_{\nu^\beta}))}\\
&\lesssim \frac{\nu^{2\beta}}{\ell}\|{u^\nu}\|_{L^\infty(0, T; L^\infty(\Omega_{\nu^\beta}))}\Vert \nabla u^\nu\|^2_{L^2(0, T; L^2(\Omega_{\nu^\beta}))}\\
&\lesssim \nu^\beta \Vert \nabla u^\nu\|^2_{L^2(0, T; L^2(\Omega_{\nu^\beta}))}=o_\nu(1)
\end{align*}
where we used again Poincar\'e's inequality \eqref{Poincare:strip} and assumption \eqref{cd:vanish}.  Since 
$$
\|\tau_\ell(u^\nu,u^\nu)\|_{L^1(0, T; L^1(\Omega_{h}))}\lesssim  \Vert  u^\nu\|^2_{L^2(0, T; L^2(\Omega_{\nu^\beta}))},
$$
we have by the same argument that
\begin{align*}
\int_0^t\int_\Omega|\nabla \theta_{\ell,h} \cdot \big( \ol{(u^\nu)}_\ell\cdot \tau_\ell(u,u)\big)|dxds
&\lesssim \nu^\beta \Vert \nabla u^\nu\|^2_{L^2(0, T; L^2(\Omega_{\nu^\beta}))}=o_\nu(1).
\end{align*}
Finally, we estimate the contribution to the boundary production due to the pressure
\begin{align*}
\int_0^t\int_\Omega|\nabla \theta_{\ell,h} \cdot \big(\ol{(p^\nu)}_\ell \ol{(u^\nu)}_\ell \big)|dxds
&\leq \int_0^t\int_{\Omega_h\setminus\Omega_{h-\ell}}|\eta_{\ell,h}'(d(x))| |\ol{(u^\nu)}_\ell| |\ol{(p^\nu)}_\ell|dxds\\
&\lesssim \frac{1}{\ell}\|p^\nu\|_{L^2(0, T; L^\infty(\Omega_{\nu^\beta}))}\Vert  \ol{(u^\nu)}_\ell \|_{L^2(0, T; L^1(\Omega_h\setminus\Omega_{h-\ell}))}\\
&\lesssim \frac{1}{\sqrt{\ell}}\|p^\nu\|_{L^2(0, T; L^\infty(\Omega_{\nu^\beta}))}\Vert {u^\nu}\|_{L^2(0, T; L^2(\Omega_{\nu^\beta}))}\\
&\lesssim \|p^\nu\|_{L^2(0, T; L^2(\Omega_{\nu^\beta}))}\nu^{\beta/2}\Vert \nabla {u^\nu}\|_{L^2(0, T; L^2(\Omega_{\nu^\beta}))}\\
&\le \delta\|p^\nu\|^2_{L^2(0, T; L^2(\Omega_{\nu^\beta}))}+\frac{C}{\delta}\nu^\beta\Vert \nabla {u^\nu}\|_{L^2(0, T; L^2(\Omega_{\nu^\beta}))}^2
\end{align*}
for any $\delta>0$.

Letting $\nu\to 0$ and then $\delta\to 0$  we obtain 
\[
 \lim_{\nu\to 0}  \int_0^t  \int_{\Omega} B_{\ell,h}^\nu[u^\nu(s)]dx ds\Bigg|_{\ell,h\sim\nu^\beta} =0,\quad t\in [0, T].
\]
The proof of vanishing of global viscous dissipation and convergence to weak Euler solution in Sections \ref{secEnCons} and \ref{SecWeakE} carries over.

\section{Proof of Theorem \ref{thm:inviscid}}\label{thm4pf}

Our proof employs the relative energy method  (see \cite{W17} for a nice review) and is given, for simplicity, in two-dimensions\footnote{By inspection of the proof, clearly the statement holds also in all dimensions $d\geq2$ by modifying the approximating sequence  $v^h$ defined in $2d$ by Eqn. \eqref{vh}. For example, in three dimensions, one takes the approximating sequence to instead be  $v^h(x,t) = \nabla \times (\theta_h(x) \psi(x,t))$ where $u$ is the strong Euler solution and $\psi$ is any vector potential satisfying $u=\nabla \times \psi$ and $\hat{n}\cdot\nabla \times \psi=0$.  Dimension $d>3$ are treated similarly.}.  
First since $v$ is divergence free, $C^1$ and satisfies $\hat{n}\cdot u=0$, then by Poincar\'{e}'s lemma there exists a stream function $\psi\in C([0,T];C^2(\ol{\Omega}))\cap C^1(\ol{\Omega}\times[0,T])$ with $\psi=0$ on $\partial\Omega$ and such that $v=\nabla^\perp\psi$.  Define an approximating sequence
\be\label{vh}
v^h(x,t)=\nabla^\perp[\theta_h(x) \psi(x,t)], \qquad  \theta_h(x)= \eta\left(\frac{d(x)}{h}\right)
\ee
where  $\eta:\mathbb{R}\to [0,1]$ is a smooth non-decreasing function which is $0$ on $y\in(-\infty,1/2]$ and $1$ on $y\in [1,\infty)$.  Note that ${\rm supp}(\theta)\subset \Omega\setminus\Omega_{h/2}$.  The functions $v^h$ are $C^1(\Omega)$, compactly supported and divergence free by construction.  Therefore, they qualify as test functions in the weak formulation of Navier-Stokes.  Let $u^\nu$ be any Leray solution of the Navier-Stokes equations  with initial data $u_0^\nu$ and consider the relative energy. 
\begin{align} \nonumber
\frac{1}{2} \|u^\nu(t) - v^h(t)\|_{L^2(\Omega)}^2&= \frac{1}{2} \|u^\nu(t)\|_{L^2(\Omega)}^2+\frac{1}{2} \|v^h(t)\|_{L^2(\Omega)}^2 - \int_\Omega u^\nu (t) \cdot v^h(t)dx \\ \nonumber
&=\frac{1}{2} \|u^\nu(t)\|_{L^2(\Omega)}^2+\frac{1}{2} \|v^h(t)\|_{L^2(\Omega)}^2 -   \int_\Omega u_0^\nu  \cdot v^h(0) dx \\
&\qquad - \int_0^t \int_\Omega \Big(\partial_t  v^h\cdot u^\nu + \nabla v^h:(u^\nu \otimes u^\nu )   + \nu \Delta v^h \cdot  u^\nu \Big) dxds 
\end{align}
a.e. $t\in [0,T]$, where we used the weak formulation of Navier-Stokes in passing to the second equality.  By the Leray energy inequality, we have $\frac{1}{2} \|u^\nu(t)\|_{L^2(\Omega)}^2\leq \frac{1}{2} \|u_0^\nu\|_{L^2(\Omega)}^2$ so that 
\begin{align} \nonumber
\frac{1}{2} \|u^\nu(t) - v^h(t)\|_{L^2(\Omega)}^2
&\leq \frac{1}{2} \|v^h(t)\|_{L^2(\Omega)}^2 + \frac{1}{2} \|u_0^\nu\|_{L^2(\Omega)}^2 -   \int_\Omega u_0^\nu  \cdot v^h(0) dx  \\
&\qquad -\int_0^t \int_\Omega \Big(\partial_t v^h\cdot u^\nu + \nabla v^h:(u^\nu \otimes u^\nu )   + \nu \Delta v^h \cdot  u^\nu \Big) dx ds.
\end{align}
Now, using the fact \eqref{distDeriv}, we have that $\nabla^\perp d(x) = -\hat{\tau}(\pi(x))$ for all $x\in \Omega_{h(\Omega)}$  so
\be\label{vhDec}
v^h(x,t) = \theta_h(x) v(x,t) -\eta'\left(\frac{d(x)}{h}\right)   \frac{ \hat{\tau}(\pi(x))}{h} \psi(x,t)
\ee
where $\hat\tau$ is the unit tangent vector field to the boundary.  From \eqref{vhDec} and the fact that $v$ is a strong Euler solution, we have
\begin{align} \nonumber
\partial_t v^h  &= \theta_h \partial_t v - \eta'\left(\frac{d(x)}{h}\right)   \frac{ \hat{\tau}(\pi(x))}{h}\partial_t\psi\\ \nonumber
&=- \theta_h \nabla\cdot(v\otimes v) -\theta_h  \nabla  p  - \eta'\left(\frac{d(x)}{h}\right)   \frac{ \hat{\tau}(\pi(x))}{h}\partial_t\psi\\
&=- \theta_h \nabla\cdot(v\otimes v) -  \nabla(\theta_h  p)  - \frac{1}{h}  \eta'\left(\frac{d(x)}{h}\right)\left( \hat{\tau}(\pi(x))\partial_t\psi + \hat{n}(\pi(x)) p \right).
\end{align}
The relative energy inequality becomes
\begin{align}\nonumber
\frac{1}{2} \|u^\nu(t) - v^h(t)\|_{L^2(\Omega)}^2
&\leq \frac{1}{2} \|v^h(t)\|_{L^2(\Omega)}^2 + \frac{1}{2} \|u_0^\nu\|_{L^2(\Omega)}^2 -   \int_\Omega u_0^\nu  \cdot v^h(0) dx  \\ \nonumber
&\quad + \int_0^t\int_\Omega \theta_h \Big(  \nabla\cdot(v\otimes v) \cdot u^\nu  - \nabla v:(u^\nu \otimes u^\nu ) \Big) dxds \\ \nonumber
&\quad  +  \int_0^t\int_{\Omega_h\setminus \Omega_{h/2}} \frac{1}{h}  \eta'\left(\frac{d(x)}{h}\right)\Big(\left( \hat{\tau}(\pi(x))\partial_t\psi + \hat{n}(\pi(x)) (p+u^\nu\otimes u^\nu \right) \Big)\cdot u^\nu dxds \\
&\quad   -\nu  \int_0^t\int_\Omega  \Delta v^h \cdot  u^\nu dx ds. \label{relenIneq1}
\end{align}
We first rewrite the second term in the inequality \eqref{relenIneq1}.  Using the divergence free properties $\nabla \cdot u^\nu= \nabla \cdot v=0$, we have after minor manipulation that
\begin{align}\nonumber
 \nabla\cdot(v\otimes v) \cdot u^\nu  - \nabla v:(u^\nu \otimes u^\nu )&= v\cdot \nabla v \cdot u^\nu  -u^\nu\cdot  \nabla v\cdot u^\nu = (v-u^\nu)\cdot \nabla v \cdot u^\nu \\ \nonumber
  &=(v-u^\nu)\cdot \nabla v \cdot (u^\nu-v) + \nabla\cdot\left(  (v-u^\nu)\frac{1}{2} |v|^2\right).
\end{align}
Upon integrating in space against $\theta_h$, this yields
\begin{align}\nonumber
\int_\Omega \theta_h \Big(  \nabla\cdot(v\otimes v) \cdot u^\nu  - & \nabla v:(u^\nu \otimes u^\nu ) \Big) dx= \int_\Omega \theta_h \Big((u^\nu - v )\cdot \nabla v \cdot  (v- u^\nu) \Big) dx\\
& \qquad +    \frac{1}{h} \int_{\Omega_h\setminus \Omega_{h/2}}   \eta'\left(\frac{d(x)}{h}\right) \hat{n}(\pi(x)) \cdot  (v(x)-u^\nu(x)) \frac{1}{2} |v(x)|^2 dx. \label{secondintterm}
\end{align}
Furthermore, since $\|\theta_h\|_{L^\infty(\Omega)}=1$, we have the bound
\be
\left|\int_0^t \int_\Omega \theta_h \Big((u^\nu - v )\cdot \nabla v \cdot  (v- u^\nu) \Big) dx ds\right|\leq  \int_0^t \|\nabla v\|_{L^\infty(\Omega)} \|u^\nu(s) - v(s)\|_{L^2(\Omega)}^2 ds.
\ee
Thus, we have the following integral inequality for a.e. $t\in [0,T]$
\be \label{relEbnd}
\frac{1}{2} \|u^\nu(t) - v(t)\|_{L^2(\Omega)}^2 \leq \int_0^t \|\nabla v\|_{L^\infty(\Omega)} \|u^\nu(s) - v(s)\|_{L^2(\Omega)}^2 ds +E^{\nu,h}(t),
\ee
where we have defined the error terms
\be
E^{\nu,h}(t) := |E_1^{\nu,h}(t)| +|E_2^{\nu,h}(t)|+|E_3^{\nu,h}(t)| +|E_4^{\nu,h}(t)|,
\ee
with
\begin{align}\nonumber
E_1^{\nu,h}(t)&:=\frac{1}{2}\Big( \|u^\nu(t) - v(t)\|_{L^2(\Omega)}^2- \|u^\nu(t) - v^h(t)\|_{L^2(\Omega)}^2  \\
&\qquad\qquad\qquad\qquad \qquad +  \|v^h(t)\|_{L^2(\Omega)}^2 + \|u_0^\nu\|_{L^2(\Omega)}^2 -   2\int_\Omega u_0^\nu  \cdot v^h(0) dx \Big),\\
E_2^{\nu,h}(t)&:= \int_0^t\int_{\Omega_h\setminus \Omega_{h/2}} \frac{1}{h}  \eta'\left(\frac{d(x)}{h}\right)\Big(\left( \hat{\tau}(\pi(x))\partial_t\psi + \hat{n}(\pi(x)) (p+u^\nu\otimes u^\nu \right) \Big)\cdot u^\nu dxds,\\
E_3^{\nu,h}(t)&:= \frac{1}{h}  \int_0^t \int_{\Omega_h\setminus \Omega_{h/2}}   \eta'\left(\frac{d(x)}{h}\right) \hat{n}(\pi(x)) \cdot  (v(x)-u^\nu(x)) \frac{1}{2} |v(x)|^2 dxds,\\
E_4^{\nu,h}(t)&:=\nu  \int_0^t\int_\Omega  \Delta v^h \cdot  u^\nu dx ds.
\end{align}
Applying Gr\"{o}nwall's inequality to \eqref{relEbnd} gives
\be\label{GronBnd}
\frac{1}{2} \|u^\nu(t) - v(t)\|_{L^2(\Omega)}^2\leq  E^{\nu,h}(t) + \int_0^t E^{\nu,h}(s) \|\nabla v(s)\|_{L^\infty(\Omega)} \exp\left(\int_{s}^t  \|\nabla v(s')\|_{L^\infty(\Omega)}  ds'\right)ds
\ee
for a.e. $t\in [0,T]$.
We will show that $E^{\nu,h}(t)\to 0$ strongly in $L^\infty([0,T])$ as $\nu,h\to 0$ with $h=\nu$.\\

\noindent \textbf{Error term $E_1^{\nu,h}(t)$.} First note that since $\psi$ is $C([0,T];C^2(\ol{\Omega}))$ and $\psi\vert_{\partial\Omega}=0$ we have
\be\label{psieps}
|\psi(x, t)|\le C|h|\quad\forall x\in \Omega_h.
\ee
In view of the formula \eqref{vhDec}, it follows easily that
\be
v^h\to v \quad \text{strongly in }\quad  L^\infty(0,T;L^2(\Omega)).
\ee
Combined with the fact that $v$ conserves energy, we have
\be
\lim_{h,\nu\to 0} \sup_{t\in [0,T]} \left| \|v^h(t)\|_{L^2(\Omega)}^2 + \|u_0^\nu\|_{L^2(\Omega)}^2 -   2\int_\Omega u_0^\nu  \cdot v^h(0) dx\right|=0. 
\ee
 with the limit $h,\nu \to 0$ taken in any order.
For the remaining terms, the reverse triangle inequality allows us to bound
\begin{align*}\nonumber
& \sup_{t\in[0,T]}\left|\|u^\nu(t) - v(t)\|_{L^2(\Omega)}^2 -  \|u^\nu(t) - v^h(t)\|_{L^2(\Omega)}^2\right| \\
& \leq  \sup_{t\in[0,T]} \left|\|u^\nu(t) - v(t)\|_{L^2(\Omega)} -  \|u^\nu(t) - v^h(t)\|_{L^2(\Omega)}\right|\left(\|u^\nu(t) - v(t)\|_{L^2(\Omega)} +  \|u^\nu(t) - v^h(t)\|_{L^2(\Omega)}\right) \\ \nonumber
&\leq  2\max\{\|u^\nu\|_{L^\infty(0,T;L^2(\Omega))}, \|v^h\|_{L^\infty(0,T;L^2(\Omega))}, \|v\|_{L^\infty(0,T;L^2(\Omega))}\}\|v^h- v\|_{L^\infty(0,T;L^2(\Omega))}^2.
\end{align*}
Thus, strong convergence of $v^h\to v$ in $L^\infty(0,T;L^2(\Omega))$ and uniform boundedness of $u^\nu,v^h$ and $v$ in $L^\infty(0,T;L^2(\Omega))$  once again allows us to conclude that this term vanishes as $h\to 0$.  Thus, we have shown
\be
\lim_{\nu,h\to 0} \sup_{t\in[0,T]} |E_1^{\nu,h}(t)| =0
\ee
in any order of the limits $h,\nu\to 0$.\\

\noindent \textbf{Error term $E_2^{\nu,h}(t)$.} To bound this contribution, we will use the fact that $u^\nu|_{\partial\Omega}=0$ to write the expression in terms of increments (with one velocity pinned at the boundary).  In particular, since $\pi(x)\in \partial \Omega$ for all $x\in \Omega_{h(\Omega)}$, choosing $h<h(\Omega)$ we have that 
\[
|E_2^{\nu,h}(t)|\le \|\eta'\|_{L^\infty(\Omega)}  \frac{1}{h} \int_0^t\int_{\Omega_h\setminus \Omega_{h/2}} \left(|p|+|\partial_t\psi|  + |w^\nu(x)\otimes w^\nu(x)|\right)|w^\nu(x)| dxds
\]
with $w^\nu(x):=u^\nu(x)- u^\nu(\pi(x))$.  Now note that $p,\partial_t\psi\in L^\infty([0,T];L^\infty( \Omega))$.  Moreover, with the choice of $h=\nu$ and our assumption that $u^\nu \in L^3(0,T;L^\infty(\Omega_{\nu}))$ with norms uniformly bounded in $\nu$, we have
\begin{align}\nonumber
|E_2^{\nu,h}(t)| & \leq  \max\{\|p\|_{L^\infty(0, T; \Omega_\nu)}, \|\partial_t\psi\|_{L^\infty(0, T; \Omega_\nu)}\} \frac{C}{\nu}\int_0^t  \int_{\Omega_\nu\setminus \Omega_{\nu/2}} |w^\nu(x)|dxds\\
&\quad+\frac{C}{\nu}\int_0^t  \int_{\Omega_\nu\setminus \Omega_{\nu/2}} |w^\nu(x)|^3dxds.
\end{align}
By the equicontinuity at the boundary \eqref{inviscidLimAssum}, for any $\delta\in (0, 1)$, there exists $\rho:=\rho(\delta)$ such that for all $h<\rho$,
\be
\begin{aligned}
|E_2^{\nu,h}(t)| &\leq  \max\{\|p\|_{L^\infty(0, T; \Omega_\nu)}, \|\partial_t\psi\|_{L^\infty(0, T; \Omega_\nu)}\}\frac{C\delta }{\nu}\int_0^t \int_{\Omega_\nu}  \gamma(d(x),s) dx ds\\
&\qquad +\frac{C\delta^3 }{\nu}\int_0^t \int_{\Omega_\nu}  |\gamma(d(x),s)|^3 dxds\\
&\le  C\Big(1+T+\max\{\|p\|_{L^\infty(0, T; \Omega_\nu)}, \|\partial_t\psi\|_{L^\infty(0, T; \Omega_\nu)}\}\Big)\Big(\delta+\frac{\delta}{\nu}\int_0^t \int_{\Omega_\nu}  |\gamma(d(x),s)|^3 dxds \Big)
\end{aligned}
\ee
where we used the elementary inequality $\gamma\le \gamma^3+1$ together with the facts that $|\Omega_\nu|\le C\nu$ and  $d(x)=|x-\pi(x)|$ for all $x\in \Omega_{h(\Omega)}$. \\

Next, we note that for any  integrable function $f\in L^1_{loc}(\mathbb{R}^+)$, the following formula holds
\be\label{tubularVar}
\int_{\Omega_{h}}   f(d(x)) dx   = \int_{0}^{h}   f(y) |N(y)| dy\leq  C_\Omega \int_{0}^{h}   f(y) dy
\ee
where $N(y):=\{x\in \Omega: d(x) = y\}$ and we have used that $|N(y)|$ is bounded by a constant depending on the domain by independent of $h, \ell$.    Eq. \eqref{tubularVar} is a special case of the general formula for integration over a tubular neighborhood of level hypersurfaces, see e.g. \cite{DK04}.    Applying  formula \eqref{tubularVar} with $f=\gamma^3$ and then using property \eqref{gamBnd} of $\gamma$  we deduce that
\begin{align}\label{E2bnd}
\sup_{t\in [0,T]} |E_2^{\nu,h}(t)| &\leq C\Big(\delta+\delta\Vert\Gamma\Vert_{L^1((0, T))}+o_\nu(1)\Big),\quad h< \rho.
\end{align}
Letting  $h=\nu\to 0$ and then $\delta\to 0$ we conclude that
\be
\lim_{\nu\to 0} \sup_{t\in [0,T]} |E_2^{\nu,h}(t)| =0.
\ee
\noindent \textbf{Error term $E_3^{\nu,h}(t)$.} Note that since $\hat{n}\cdot u^\nu|_{\partial\Omega}=\hat{n}\cdot v|_{\partial\Omega}=0$, this error term can be again written in terms of boundary increments.  In particular
\begin{align} \nonumber
|E_3^{\nu,h}(t)| &\leq   \frac{1}{h}\left|  \int_0^t \int_{\Omega_h\setminus \Omega_{h/2}}   \eta'\left(\frac{d(x)}{h}\right) \hat{n}(\pi(x)) \cdot  v(x) \frac{1}{2} |v(x)|^2 dxds\right|\\
&\qquad +  \frac{1}{h}\left|  \int_0^t \int_{\Omega_h\setminus \Omega_{h/2}}   \eta'\left(\frac{d(x)}{h}\right) \hat{n}(\pi(x)) \cdot w^\nu(x) \frac{1}{2} |v(x)|^2 dxds\right|.
\end{align}
To estimate the first term, note that since $v\in C^1([0,T]\times \Omega)$ we have that $|\hat{n}(\pi(x))\cdot v(x)|\leq C d(x)\leq C h$.  For the other term, recall that we set $h=\nu$.  Therefore, since $u^\nu$ is equicontinuous at the boundary in a layer of size $\nu$,  for any $\delta>0$, there exists $\rho:=\rho(\delta)$ such that for all $h<\rho$ we have
\begin{align}\nonumber
|E_3^{\nu,h}(t)| &\leq  C \int_0^t  \| v(s)\|_{L^2(\Omega_\nu\setminus \Omega_{\nu/2})}^2 ds+ C \delta\int_0^t \| v(s)\|_{L^\infty(\Omega_\nu)}^2 \frac{1}{\nu}\int_{\Omega_\nu} \gamma(d(x),s)   dxds.
\end{align}
The first contribution above is easily seen to vanish by Lebesgue dominated convergence theorem since $v\in L^\infty(0,T;L^2(\Omega))$.  The second contribution vanishes since the integral is treated exactly as in \eqref{E2bnd} and, for $\nu\to 0$, $\delta$ is permitted arbitrarily small.  Thus,
\be
\lim_{\nu\to 0} \sup_{t\in [0,T]} |E_3^{\nu,h}(t)| =0.
\ee
\noindent \textbf{Error term $E_4^{\nu,h}(t)$.} Note that, again, we can use the no slip boundary conditions on the Navier-Stokes solution to introduce a velocity increment at the boundary
\be
\nu  \int_0^t\int_\Omega  \Delta v^h \cdot  u^\nu dx ds= \nu  \int_0^t\int_\Omega  \Delta v^h \cdot  (u^\nu(x)- u^\nu(\pi(x))) dx ds.
\ee
The Laplacian applied to the localized Euler solution involves many terms
\be
\Delta v^h=   \Delta\Big(\theta_h(x) v(x,t) +\eta'\left(\frac{d(x)}{h}\right)   \frac{ \hat{\tau}(\pi(x))}{h} \psi(x,t)\Big).
\ee
Note that under our assumption that $\partial \Omega\in C^3$, we have $\sigma\in C^2(\Omega_{h(\Omega)})$ and $\hat{n},\hat{\tau}\in C^2(\partial \Omega)$.   By inspection, it is clear that the worst term (in the sense of most powers of $1/h$) is
\be
\nu  \int_0^t\int_\Omega  \Delta v^h \cdot  u^\nu dx ds\sim   \frac{\nu}{h^3}\int_0^t\int_{\Omega_h\setminus \Omega_{h/2}}  \eta'''\left(\frac{d(x)}{h}\right)   \hat{\tau}(\pi(x)) \psi(x,t) \cdot (u^\nu(x)- u^\nu(\pi(x))) dx ds.
\ee
We estimate this using \eqref{psieps}. Then, for any $\delta>0$, there exists $\rho:=\rho(\delta)$ s.t. for all $h<\rho$ we have
\be
\nu \left| \int_0^t\int_\Omega  \Delta v^h \cdot  u^\nu dx ds\right|\sim   \frac{C}{\nu} \int_0^t\int_{\Omega_\nu} |u^\nu(x)- u^\nu(\pi(x))| dx\leq C \delta \int_0^t\frac{1}{\nu}\int_{\Omega_\nu} \gamma(d(x),s)   dxds.
\ee
This vanishes as  $\nu\to 0$ by the same argument as in \eqref{E2bnd}, since $\delta$ is permitted arbitrarily small.  We conclude 
\be
\lim_{\nu\to 0}   \sup_{t\in [0,T]}|E_4^{\nu,h}(t)| =0.
\ee
\noindent \textbf{$L^2$ convergence.} 
Altogether, we have shown that $\sup_{t\in [0,T]} |E^{\nu,h}(t)| \to 0$ as $h=\nu\to 0$.   Consequently, combined with the bound  \eqref{GronBnd},  we conclude  that
\be
\lim_{\nu\to 0}\sup_{t\in [0,T]} \|u^\nu(t) - v(t)\|_{L^2(\Omega)}^2= 0.
\ee


\subsection*{Acknowledgments}  We are grateful to C. Bardos, E. Titi and E. Wiedemann for communicating an early version of their forthcoming preprint \cite{BTW18} and interesting correspondence which resulted in improvements to the present manuscript. We thank P. Constantin, \red{ V. Vicol and  the anonymous referee for valuable comments.
} 
Research of TD is partially supported by NSF-DMS grant 1703997.

\vspace{-1.4mm}


\begin{thebibliography}{99}
\bibitem{O49} Onsager, L.: Statistical hydrodynamics. Il Nuovo Cimento (Supplemento), {\bf{6}}: 279-287 (1949)
\bibitem{GLE94} 
Eyink, G.~L.: Energy dissipation without viscosity in ideal hydrodynamics I. Fourier analysis and local energy transfer, Physica D {\bf 78}: 222--240,  (1994)
\bibitem{CET94}Constantin, P. , W. E, and Titi, E.: Onsager's conjecture on the energy conservation for solutions of Euler's equation. Comm. Math. Phys. {\bf{165}}: 207-209,  (1994)
\bibitem{DR00}
Duchon, J. and Robert, R.: Inertial energy dissipation for weak solutions of incompressible Euler and Navier-Stokes equations. Nonlinearity {\bf{13.1}}: 249--255, (2000)
\bibitem{CCFS08}
Cheskidov, A., Constantin, P., Friedlander, S. and Shvydkoy, R.: 
Energy conservation and Onsager's conjecture for the Euler equations, 
Nonlinearity {\bf 21}: 1233--1252,  (2008)

\bibitem{CFLS}
Cheskidov, A.,  Lopes Filho, M. C.,  Nussenzveig Lopes, H. J., Shvydkoy, R.: Energy conservation in two-dimensional incompressible ideal fluids. Comm. Math. Phys. 348 (2016), no. 1, 129--143.

\bibitem{Shvydkoy}
Shvydkoy, R: Lectures on the Onsager conjecture, DCDS-S Volume: 3(3), 473-496, (2010)

\bibitem{DE17}
Drivas. T.D. and Eyink G.L.: An Onsager Singularity Theorem for   Leray--Hopf solutions of Incompressible Navier-Stokes. preprint arXiv:1710.05205, (2017)
\bibitem{DE17comp}
Drivas, T.D., and Eyink, G.L.: An Onsager singularity theorem for turbulent solutions of compressible Euler equations. Comm. Math. Phys. (2017): 1-31.
\bibitem{LS09} 
De Lellis, C. and Sz\'ekelyhidi Jr., L.:  The Euler equations as a differential inclusion. {\bf 170}: 1417-1436 (2009) 
\bibitem{LS10}
De Lellis, C. and Sz\'ekelyhidi Jr., L.:  On admissibility criteria for weak solutions of the Euler equations. Arch. Ration. Mech. and Anal. {\bf 195}: 225--260, (2010) 
\bibitem{LS12}
De Lellis, C. and Sz\'ekelyhidi Jr., L.:  The $h$-principle and the equations of fluid dynamics,
B. Am. Math. Soc. {\bf 49}:  347--375,  (2012)
\bibitem{I16} 
Isett, P.: A proof of Onsager's conjecture, preprint arXiv:1608.08301,  (2016)
\bibitem{BLSV17}
Buckmaster, T., De Lellis, C., Sz\'ekelyhidi Jr., L. and Vicol, V.: 
Onsager's conjecture for admissible weak solutions, 
preprint arXiv:1701.08678,  (2017)
\bibitem{E18}
Eyink, G.L.: Review of the Onsager ``Ideal Turbulence" Theory. arXiv preprint arXiv:1803.02223 (2018).
\bibitem{BO95} 
Borue, V. and Orszag, S.~A.:
Self-similar decay of three-dimensional homogeneous turbulence with hyperviscosity
Phys. Rev. E {\bf 51}: R856--R859, (1995) 
\bibitem{TBS02}
Touil, H., Bertoglio,  J.-P. and Shao, L.:
The decay of turbulence in a bounded domain, J. Turb. {\bf 3}: N49, (2002)
\bibitem{KRS98}
Sreenivasan, K.~R.: An update on the energy dissipation rate in isotropic turbulence, 
Phys. Fluids {\bf 10}: 528--529,  (1998)
\bibitem{KIYIU03}
Kaneda, Y., Ishihara, T., Yokokawa, M., Itakura, K.,  and Uno, A.: 
Energy dissipation rate and energy spectrum in high resolution direct numerical simulations 
of turbulence in a periodic box, Phys. Fluids {\bf 15}: L21--L24,  (2003)
\bibitem{KRS84} 
Sreenivasan, K.~R.: On the scaling of the turbulence energy dissipation rate, 
Phys. Fluids {\bf 27}: 1048--1051,  (1984)
\bibitem{PKW02}
 Pearson, B.~R., Krogstad, P.~A., and van de Water, W.:
Measurements of the turbulent energy dissipation rate, Phys. Fluids {\bf 14}: 1288--1290,  (2002)
\bibitem{GLE95} 
Eyink, G.L.: Besov spaces and the multifractal hypothesis, 
J. Stat. Physics {\bf 78}: 353--375, (1995) 
\bibitem{T84} 
Temam, R.: Navier-Stokes Equations: Theory and Numerical Analysis (North-Holland,
Amsterdam, 1984)
\bibitem{L34}
Leray, J.: Sur le mouvement d'un liquide visqueux emplissant l'espace, Acta Math. {\bf 63}: 193--248, (1934)
\bibitem{H50}
Hopf, E.: \"{U}ber die Anfangswertaufgabe f\"{u}r die hydrodynamischen Grundgleichungen. Erhard Schmidt zu seinem 75. Geburtstag gewidmet. Mathematische Nachrichten 4.1-6: 213-231, (1950)
\bibitem{Jmm14}
Mauro, J.~A., On the regularity properties of the pressure field associated to a Hopf weak solution to the Navier-Stokes equations, Pliska Stud. Math. Bulgar. 23, 95--118 (2014)
\bibitem{Giga91}
Giga, Y.  and Sohr, H. Abstract $L^p$ estimates for the Cauchy problem with applications to the Navier-Stokes equations in exterior domains. Journal of Functional Analysis 102, 1, 72--94, (1991)
\bibitem{Sohr86}
 Sohr, H. and von Wahl, W. On the regularity of the pressure of weak
solutions of Navier-Stokes equations. Arch. Math. 46, 5, 428--439,  (1986)
\bibitem{BT18}
Bardos, C., and Titi, E.:  Onsager's Conjecture for the Incompressible Euler Equations in Bounded Domains. Archive for Rational Mechanics and Analysis 228.1 (2018): 197-207.
\bibitem{BTW18}
Bardos, C., Titi, E., Wiedemann, E.: Onsager's Conjecture with Physical Boundaries and an
Application to the Viscosity Limit, preprint arXiv:1803.04939, (2018)
\bibitem{LS99}
Lewis, G., and Swinney, H. Velocity structure functions, scaling, and transitions in high-Reynolds-number Couette-Taylor flow. Phys. Rev. E 59.5: 5457, (1999).
\bibitem{EscMon}  Escauriaza, L. and Montaner, S.: Some remarks on the $L^p$ regularity of second derivatives of solutions to non-divergence elliptic equations and the Dini condition, Rend. Lincei. Mat. Appl. (28): 49-63, (2017).
\bibitem{C97}
Cadot, O., Couder, Y., Daerr, A., Douady, S., and Tsinober, A.: Energy injection in closed turbulent flows: Stirring through boundary layers versus inertial stirring. Phys. Rev. E, 56(1), 427, (1997).
\bibitem{Fureby97}
Fureby, C., and G. Tabor.: Mathematical and physical constraints on large-eddy simulations. Theoretical and Computational Fluid Dynamics 9.2 (1997): 85-102.
\bibitem{Ghosal99}
Ghosal, S..: Mathematical and physical constraints on large-eddy simulation of turbulence. J. AIAA  37.4 (1999): 425-433.
\bibitem{Bos05}
van der Bos, R., and Geurts, B.: Commutator errors in the filtering approach to large-eddy simulation. Phys. Fluids 17.3 (2005): 035108.
\bibitem{Geurts06}
Geurts, B., and Holm, D.: Commutator errors in large-eddy simulation. J. Phys. A, 39.9: 2213, (2006)
\bibitem{Chen12}
Chen, S., Xia, Z., Pei, S., Wang, J., Yang, Y., Xiao, Z., and Shi, Y.: Reynolds-stress-constrained large-eddy simulation of wall-bounded turbulent flows. J. Fluid Mech., 703, 1-28, (2012)
\bibitem{Kato84}
Kato, T.: Remarks on zero viscosity limit for nonstationary Navier-Stokes flows with boundary. Seminar on nonlinear partial differential equations. Springer, New York, NY, 1984.
\bibitem{P32}
Prandtl, L.. Zur turbulenten Str\"{o}mung in R\"{o}hren und l\"{a}ngs.  Platten, Ergebn.
Aerodyn. Versuchsanst, G\"{o}ttingen 4 18-29 (1932)
\bibitem{vK30}
von K\'{a}rm\'{a}n, T.: Mechanische \"{A}hnlichkeit und Turbulenz. Nach. Ges. Wiss. 
G\"{o}ttingen, Math.-Phys. Kl. 58-76 (1930)
\bibitem{Enotes} Eyink,  G.L. Turbulence Theory. course notes, \url{http://www.ams.jhu.edu/~eyink/Turbulence}, (2015)
\bibitem{S07}
Sreenivasan, K. R.: Scaling and structure in high Reynolds number wall-bounded flows. Ed. Bev. J. Mck. Royal Soc., (2007).
\bibitem{P07}
Panton, R. L.: Composite asymptotic expansions and scaling wall turbulence. Philosophical Transactions of the Royal Society of London A: Mathematical, Physical and Engineering Sciences 365.1852 (2007): 733-754.
\bibitem{Farge11}
Nguyen van yen, R., Farge, M., and Schneider., K.: Energy dissipating structures produced by walls in two-dimensional flows at vanishing viscosity. Phys. Rev. Lett. 106.18 (2011): 184502.
\bibitem{CV17}
Constantin, P., and Vicol, V.: Remarks on high Reynolds numbers hydrodynamics and the inviscid limit.  preprint arXiv:1708.03225 (2017).
\bibitem{S96}
Sreenivasan, K. R., Vainshtein, S. I., Bhiladvala, R., San Gil, I., Chen, S., and Cao, N.: Asymmetry of velocity increments in fully developed turbulence and the scaling of low-order moments. Phys. Rev. Letts. 77(8), 1488,  (1996).
\bibitem{A84}
Anselmet, F., Gagne, Y., Hopfinger, E. J., and Antonia, R. A.: High-order velocity structure functions in turbulent shear flows. Journal of Fluid Mechanics, 140, 63-89, (1984)
\bibitem{CEIV17}
Constantin, P., Elgindi, T., Ignatova, M., and Vicol, V.: Remarks on the Inviscid Limit for the Navier--Stokes Equations for Uniformly Bounded Velocity Fields. SIAM Journal on Mathematical Analysis, 49(3), 1932-1946 (2017)
\bibitem{W17}
Wiedemann, E..: Weak-strong uniqueness in fluid dynamics. arXiv preprint arXiv:1705.04220 (2017).
\bibitem{B14}
Bardos, C., Székelyhidi Jr, L. and  Wiedemann, E.: Non-uniqueness for the Euler equations: the effect of the boundary.  Russian Mathematical Surveys 69.2 (2014): 189.
\bibitem{Evans10}
 Evans, L. C., Partial Differential Equations, Graduate Studies in Mathematics, American Mathematical Society, (2010).
\bibitem{DK04}
Duistermaat, J., and Kolk, J.: Multidimensional real analysis I: differentiation. Vol. 86. Cambridge University Press, 2004.
\bibitem{Isett17}
Isett, P.: Nonuniqueness and existence of continuous, globally dissipative Euler flows. preprint arXiv:1710.11186 (2017).
\end{thebibliography}
\end{document}